\documentclass[12pt,leqno]{amsart}
\usepackage{amsmath,stmaryrd}
\usepackage{amssymb}
\usepackage{amsthm}
\usepackage{epsf}
\usepackage{graphicx}
\usepackage{esint} 
\usepackage{dsfont}
\usepackage[pagebackref]{hyperref} 
\hypersetup{pdfpagemode=FullScreen, colorlinks=true, citecolor={blue}} 
\usepackage{enumerate} 

\graphicspath{ {./pictures/} }

\numberwithin{equation}{section}

\renewcommand{\arraystretch}{1.4}

\DeclareMathOperator{\dist}{dist}
\DeclareMathOperator{\supp}{supp}

\newcommand{\1}{{\mathds 1}}

\newcommand{\ms}{\medskip}

\newcommand{\bp}{\noindent {\em Proof: }}
\newcommand{\ep}{\hfill $\square$ \medskip}

\theoremstyle{plain}
\newtheorem{theorem}{Theorem}[section]
\newtheorem{lemma}[theorem]{Lemma}
\newtheorem{corollary}[theorem]{Corollary}
\newtheorem{proposition}[theorem]{Proposition}

\theoremstyle{definition}

\theoremstyle{remark}
\newtheorem{remark}[theorem]{Remark}

\begin{document}

\title[The reverse Riesz transforms is unbounded in spaces with slow diffusion]{In spaces with a slow diffusion, the Riesz transform is unbounded on $L^p$, $p\in (2,\infty)$.}

\author[Feneuil]{Joseph Feneuil}
\address{Joseph Feneuil. Universit\'e Paris-Saclay, Laboratoire de math\'ematiques d’Orsay, 91405, Orsay, France}
\email{joseph.feneuil@universite-paris-saclay.fr}

\begin{abstract} 
In graphs and Riemannian manifolds where the kernel of the diffusion semigroup satisfies pointwise sub-Gaussian estimates, we study the range of parameters \( p \in (1, \infty) \) and \( \gamma \in [0, 1] \) for which the quantities \( \|\Delta^\gamma f\|_p \) and \( \|\nabla f\|_p \) can be compared. In particular, we prove that in such metric spaces, the Riesz transform \( \nabla \Delta^{-1/2} \) is unbounded on \( L^p \) for all \( p \in (2, \infty) \), thereby demonstrating a clear departure from the behavior observed in the Euclidean setting.
\end{abstract}

\maketitle

\ms\noindent{\bf Keywords:} Riesz transform, reverse Riesz transform, quasi-Riesz transform,  sub-Gaussian heat kernel estimates, graphs, Riemannian manifolds.

\ms\noindent

\tableofcontents

\section{Introduction}
In \( \mathbb{R}^n \), the Fourier transform \( \mathcal{F} \) is used to define the Riesz transform as
\[
\mathcal{R} f := \mathcal{F}^{-1} \left[ \frac{i\xi}{|\xi|} \mathcal{F}f \right].
\]
Parseval's identity immediately implies that the Riesz transform \( \mathcal{R} \) is an isometry on \( L^2(\mathbb{R}^n) \). For smooth functions, \( \mathcal{R} \) can also be expressed as the singular integral
\[
\mathcal{R} f(x) := \text{p.v.} \int_{\mathbb{R}^n} \frac{x - y}{|x - y|^{n+1}} f(y) \, dy.
\]
Using Calderón–Zygmund theory, combined with duality arguments, one obtains (see \cite[Chapter 2, Theorem 1]{Stein70}) that
\begin{equation} \label{Ep} \tag{E$_p$}
C_p \|f\|_p \leq \|\mathcal{R} f\|_p \leq C_p \|f\|_p \qquad \text{for } p \in (1, \infty),
\end{equation}
and this method can be easily adapted to establish the same inequality on Lie groups.

If we define \( (-\Delta)^{-1/2} \) on a dense subdomain of \( L^2(\mathbb{R}^n) \), such as \( C^\infty_0(\mathbb{R}^n) \), via spectral theory, we observe from the Fourier transform definition that
\[
\mathcal{R} f = \nabla (-\Delta)^{-1/2} f \qquad \text{for } f \in C^\infty_0(\mathbb{R}^n).
\]
This identity is particularly useful, as it allows us to define the Riesz transform in a more general setting. Indeed, suppose we are given a metric measure space \( (\Gamma, d, m) \), equipped with a non-negative self-adjoint operator \( \Delta \) on \( L^2(\Gamma, m) \) with dense domain. We can define a Dirichlet form
\[
\mathcal{E}(f, g) := -\int_\Gamma f \Delta g \, dm,
\]
and a ``gradient length"
\[
\nabla f := \mathcal{E}(f, f).
\]
Two primary examples arise naturally: 
\begin{itemize}
    \item If \( \Gamma \) is a Riemannian manifold, then \( \Delta \) is the (positive) Laplace–Beltrami operator and \( \nabla f \) corresponds to the Riemannian gradient.
    \item If \( \Gamma \) is a graph equipped with a Markov operator \( P \) describing a random walk, then \( \Delta := I - P \), and \( \nabla f(x) \) measures the norm of the difference between \( f(x) \) and the values of \( f \) at its neighboring points.
\end{itemize}
The spaces we consider in this work - namely graphs and Riemannian manifolds, representing discrete and continuous settings respectively - are formally introduced in Section~\ref{Sspaces}.

\medskip

In our context, the identity \( \|\nabla f\|_2 = \|\Delta^{1/2} f\|_2 \) holds automatically by construction. We are interested in the existence of a constant \( C = C(p, \Gamma) \) such that the following inequalities hold:
\begin{equation} \label{Rp} \tag{R$_{p,\frac{1}{2}}$}
\|\nabla f\|_{L^p} \leq C \|\Delta^{1/2} f\|_{L^p},
\end{equation}
or
\begin{equation} \label{RRp} \tag{RR$_{p,\frac{1}{2}}$}
\|\Delta^{1/2} f\|_{L^p} \leq C \|\nabla f\|_{L^p},
\end{equation}
for a given \( p \in (1, \infty) \), where \( f \) is chosen from an appropriate dense subdomain. These questions have attracted significant attention since the 1980s, following a question posed by Strichartz in \cite{Str83}, asking for which Riemannian manifolds and which values of \( p \) the inequalities \eqref{Rp} and \eqref{RRp} hold. We briefly list some notable related results:

\begin{enumerate}[(i)]
    \item Inequality \eqref{Rp} holds for \( 1 \leq p \leq 2 \) on doubling Riemannian manifolds that satisfy pointwise Gaussian heat kernel bounds (see \cite{CD99}), and on graphs where the associated Markov process satisfies analogous bounds (see \cite{Russ00}).
    
    \item The result in (i) extends to sub-Gaussian heat kernel bounds on Riemannian manifolds and graphs; see \cite{CCFR15}.
    
    \item Inequality \eqref{Rp} holds for \( 2 < p \leq 2 + \varepsilon \) on doubling Riemannian manifolds (see \cite{ACDH04, AC05}) and on graphs (see \cite{BR09}) that satisfy an \( L^2 \) Poincaré inequality on balls. See also \cite{RD22} for results linking the reverse Riesz transform \eqref{RRp} and Poincaré inequalities.
    
    \item The assumptions in \cite{ACDH04, BR09} are weakened in \cite{BF15}.
    
    \item The Riesz transform is bounded in certain Hardy spaces \( H^1 \) (i.e., the endpoint case \( p = 1 \)); see \cite{Russ01, AMR08, HLMMY11, BD14, Fen15}.
\end{enumerate}

\medskip

In \cite{CCFR15}, we provided the complete range of validity for the estimates \eqref{Rp} and \eqref{RRp} on Vicsek graphs, Vicsek cable systems, and Vicsek manifolds. An example of a Vicsek graph is shown in Figure~\ref{fig1}. Vicsek cable systems and Vicsek manifold are metric spaces and manifolds, respectively, that satisfy the same property as the Vicsek graphs at non-local scales.

\begin{figure}[!ht] 
\centering
\caption{Vicsek graph (4th step of the construction) with volume growth \( V(x,r) \approx r^{\log_3 5} \) and escape rate \( r^{\log_3 5 + 1} \).}
\label{fig1}
\includegraphics[width=0.33\textwidth]{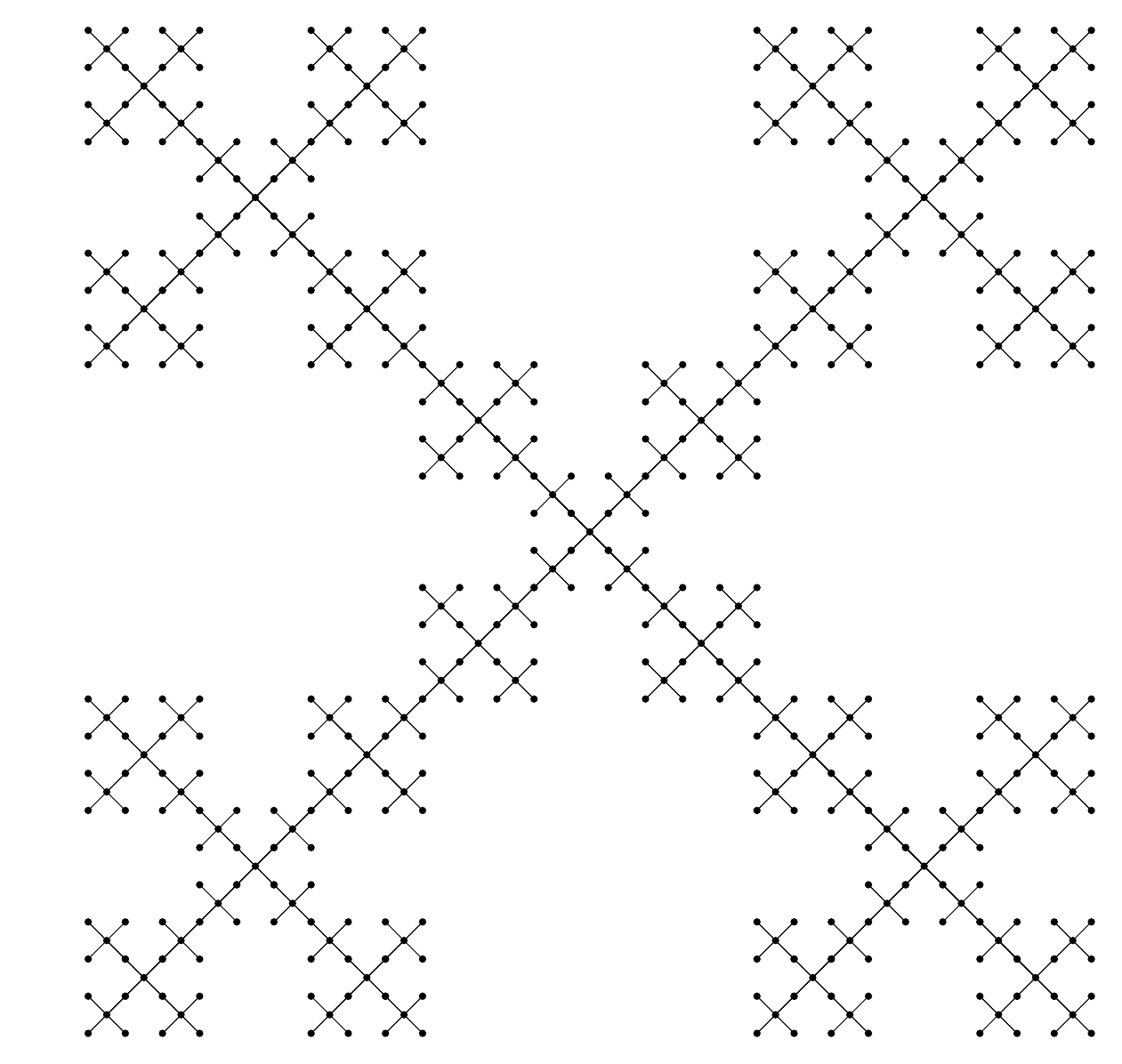}
\end{figure}

\begin{theorem}[{\cite[Theorem 5.3]{CCFR15}}] \label{ThCCFR}
For Vicsek graphs and Vicsek manifolds, the bound \eqref{Rp} holds if and only if \( p \in (1,2] \), and  \eqref{RRp} holds if and only if \( p \in [2,\infty) \).
\end{theorem}

Vicsek graphs are in a sense extreme: they exhibit the slowest possible diffusion, or equivalently, the minimal number of paths for a given volume growth rate. This raises a natural question: is the mutual exclusivity of the Riesz and reverse Riesz transform estimates - i.e., the fact that one inequality fails when the other holds, and vice versa, except at \( p=2 \) - a phenomenon specific to this extreme case, or does it reflect a more general property of graphs lacking pointwise Gaussian upper bounds (also referred to as "fractal-type" graphs)?

The purpose of this article is to establish the latter. We show that if the Markov or heat kernel satisfies a pointwise sub-Gaussian estimate, then the Riesz transform is unbounded for all \( p \in (2,\infty) \). More generally, we argue that the Riesz transform is not the appropriate tool to study in the sub-Gaussian setting. Instead of comparing \( \nabla f \) with \( \Delta^{1/2} f \) in \( L^p \), it is more meaningful to compare \( \nabla f \) with \( \Delta^\gamma f \) in \( L^p \) for an appropriate choice of $\gamma$ that depends on $p$.

When the ambient space \( \Gamma \) is a graph, we consider, for $f\in L^p(\Gamma)$ and $C>0$ independent of $f$, the inequalities
\begin{equation} \label{Rpgg} \tag{R$_{p,\gamma}$}
\|\nabla f\|_{L^p} \leq C \|\Delta^{\gamma} f\|_{L^p}
\end{equation}
and
\begin{equation} \label{RRpgg} \tag{RR$_{p,\gamma}$}
\|\Delta^{\gamma} f\|_{L^p} \leq C \|\nabla f\|_{L^p}.
\end{equation}
When the ambient space is a Riemannian manifold, we instead consider versions with an extra smoothing, we want the existence of $C>0$ such that for any $f\in C^\infty_0(\Gamma)$, we have either
\begin{equation} \label{Rpgm} \tag{R$_{p,\gamma}$}
\|\nabla e^{-\Delta} f\|_{L^p} \leq C \|\Delta^{\gamma} f\|_{L^p}
\end{equation}
or
\begin{equation} \label{RRpgm} \tag{RR$_{p,\gamma}$}
\|\Delta^{\gamma} e^{-\Delta} f\|_{L^p} \leq C \|\nabla f\|_{L^p}.
\end{equation}

Why this difference? In the discrete setting, both \( \Delta \) and \( \nabla \) are bounded operators that satisfy
\[
\|\Delta f\|_p \leq C \|\nabla f\|_p \leq C' \|f\|_p \qquad \text{for } f \in L^p(\Gamma),
\]
(see Lemma~\ref{lemRp1}), making it possible for inequalities like \eqref{Rpgg} and \eqref{RRpgg} to hold even when \( \gamma \neq 1/2 \). In contrast, smooth Riemannian manifolds are locally flat (i.e., locally Euclidean), so \( \|\nabla f\|_p \) is naturally compared only to \( \|\Delta^{1/2} f\|_p \). To focus on large-scale properties of the manifold and suppress small-scale variations, we regularize \( f \) using the heat semigroup.

\medskip

A partial result for Vicsek cable systems was established by Devyver and Russ:

\begin{theorem}[{\cite[Theorem 1.8]{DR24}}] \label{ThDR}
Let \( \Gamma \) be a Vicsek cable system with escape rate \( r^\beta \), which is a particular case of metric space satisfying the pointwise estimate \eqref{UEm} and Ahlfors regular condition \( C^{-1} r^{\beta-1} \leq V(x,r) \leq C r^{\beta - 1} \) when \(x\in \Gamma\) and \( r \geq 1 \). Then:
\begin{enumerate}[(i)]
\item The inequality \eqref{RRpgm} holds when \( p \in [2,\infty) \) and \( \gamma \geq \frac{1}{2} \), and also when \( p \in (1,2) \) and \( \gamma > \frac{1}{\beta} + \frac{1}{p} \left(1 - \frac{2}{\beta} \right) \);
\item The inequality \eqref{RRpgm} fails when \( p \in (1,\infty) \) and \( \gamma < \frac{1}{\beta} + \frac{1}{p} \left(1 - \frac{2}{\beta} \right) \).
\end{enumerate}
By duality, it follows that \eqref{Rpgm} fails when \( p \in (1,\infty) \) and \( \gamma > \frac{1}{\beta} + \frac{1}{p} \left(1 - \frac{2}{\beta} \right) \).
\end{theorem}

A visualization of the validity region for \eqref{RRpgm} in Theorem~\ref{ThDR} is shown in Figure~\ref{figRpV}.

\begin{figure}[!ht] 
\centering
\caption{Theorem~\ref{ThDR}: Validity region for \eqref{RRpgm} on Vicsek cable systems. \vspace{-.3cm} \\ 
White indicates where the property holds, dark indicates where the property fails, and light gray indicates that we do not know.}
\label{figRpV}
\includegraphics[width=0.4\textwidth]{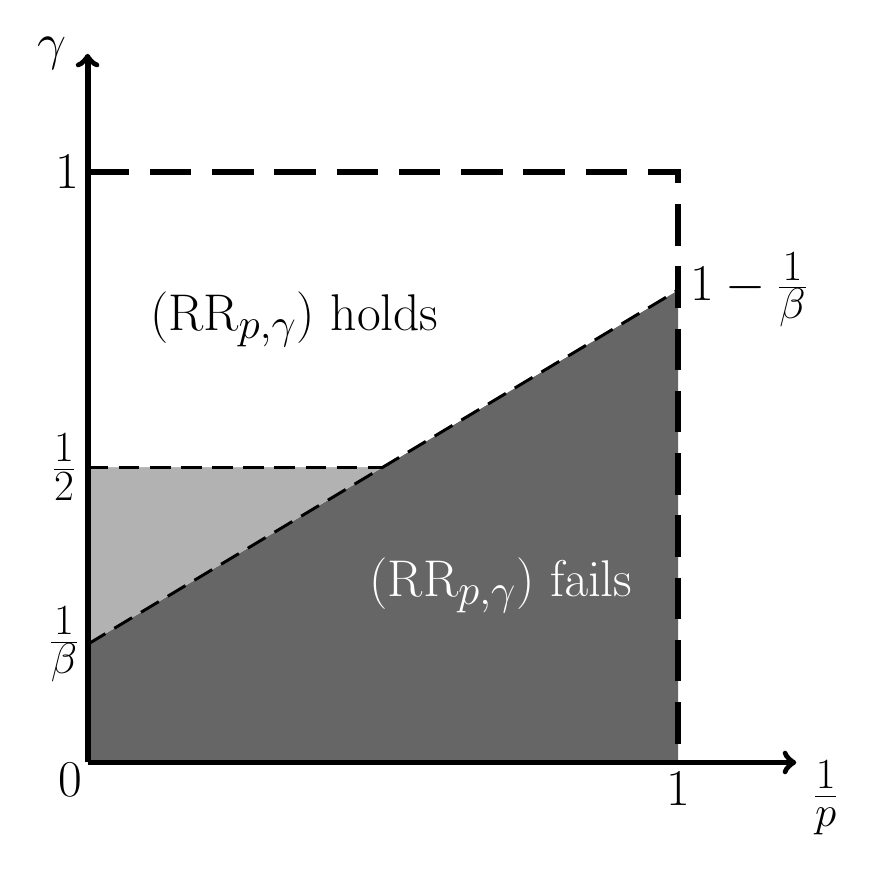}
\end{figure}

\medskip

One particularly interesting direction is to complete the picture of where estimates \eqref{Rpgm} and \eqref{RRpgm} hold for Vicsek cable systems. Equally compelling is the question of what happens for other graphs and Riemannian manifolds exhibiting fractal-type structures. In particular, as already suggested in the statement of Theorem~\ref{ThDR}, the escape rate should play a crucial role in determining the range of \( p \in (1, \infty) \) and \( \gamma \in [0,1] \) for which the estimates \eqref{Rpgm} and \eqref{RRpgm} are valid.

\medskip

Our main result\footnote{All definitions - especially that of graphs - are provided in Section~\ref{Sspaces}.} in the case where the space \( \Gamma \) is a graph is the following:

\begin{theorem} \label{MainTh}
Let \( (\Gamma,\mu) \) be a graph whose Markov semigroup is analytic and has non-degenerate transition probabilities. Assume that \( \Gamma \) satisfies \eqref{D} and \eqref{UE} for some \( \beta \geq 2 \). Then the following holds:
\begin{enumerate}[i.]
    \item \eqref{Rpgg} holds when \( p \in (1,\infty) \) and \( \gamma \leq \min\left\{\frac{1}{2}, \frac{1}{p} \right\} \),
    \item \eqref{RRpgg} holds when \( p \in (1,\infty) \) and \( \gamma \geq \max\left\{\frac{1}{2}, \frac{1}{p} \right\} \),
    \item \eqref{RRpgg} fails when \( p \in (1,\infty) \) and \( \gamma < \min\left\{ \max\left\{ \frac{1}{p}, \frac{1}{\beta} \right\}, \frac{1}{\beta} + \frac{1}{p} \left( 1 - \frac{2}{\beta} \right) \right\} \),
    \item \eqref{Rpgg} fails when \( p \in (1,\infty) \) and \( \gamma > \max\left\{ \min\left\{ \frac{1}{p}, 1 - \frac{1}{\beta} \right\}, \frac{1}{\beta} + \frac{1}{p} \left( 1 - \frac{2}{\beta} \right) \right\} \).
\end{enumerate}
In particular, when \( \beta > 2 \), estimate \eqref{Rp} holds if and only if for \( p \in (1,2] \), and estimate \eqref{RRp} holds if and only if \( p \in [2,\infty) \).
\end{theorem}

\begin{figure}[!ht] 
\centering
\caption{Theorem~\ref{MainTh}: Validity region of \eqref{Rpgg} and \eqref{RRpgg}.}
\label{figRp}
\includegraphics[width=0.8\textwidth]{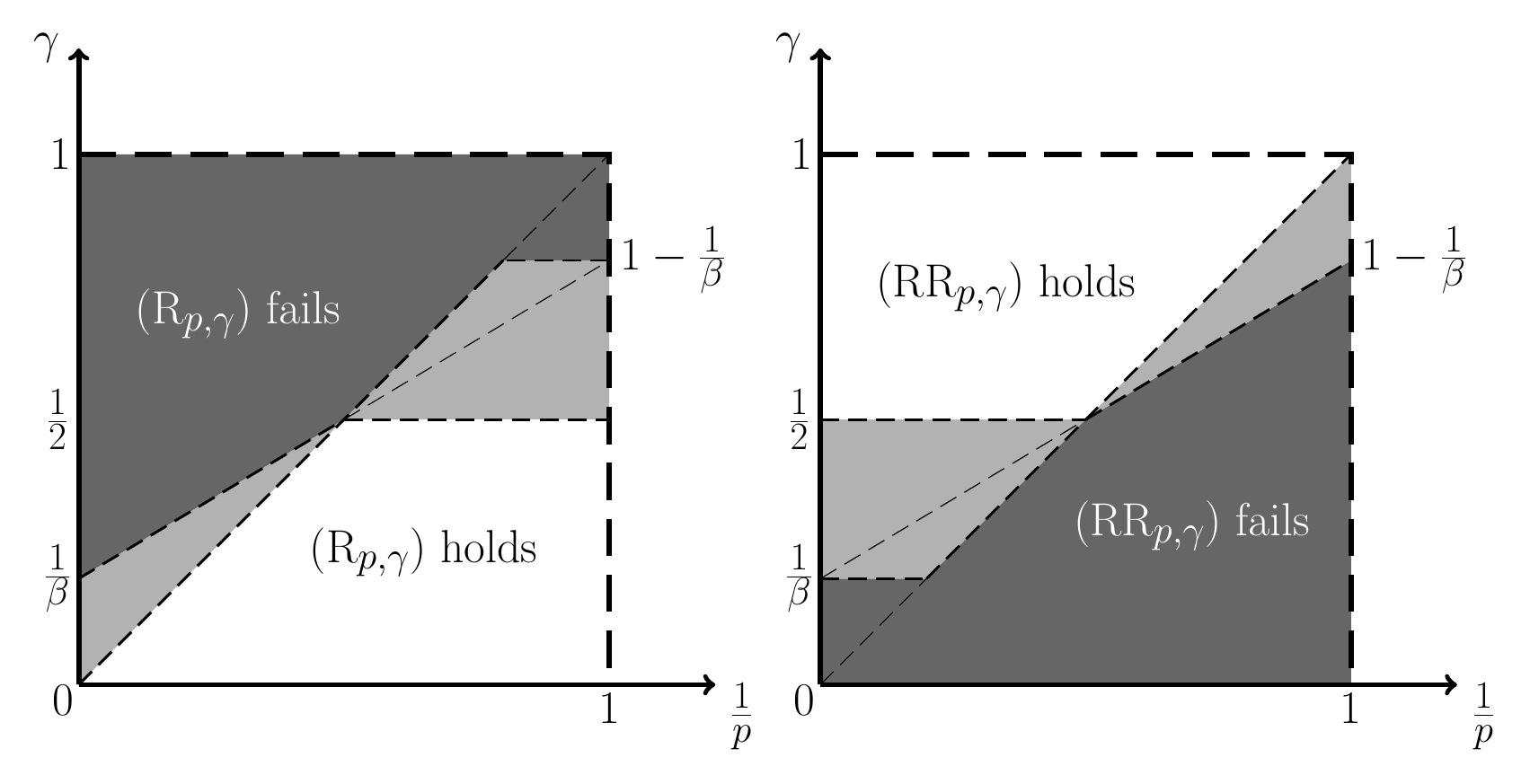}
\end{figure}

Note that items \textit{ii.} and \textit{iv.} in Theorem~\ref{MainTh} are obtained by duality (see Theorem~\ref{ThDuality}) from \textit{i.} and \textit{iii.}, respectively. The area of validity for \eqref{Rpgg} and \eqref{RRpgg} is illustrated in Figure~\ref{figRp}. The proof of the theorem combines several smaller results, and the structure of the argument is depicted in Figure~\ref{figRpproof}.

\begin{figure}[!ht] 
\centering
\caption{Structure of the proof of Theorem~\ref{MainTh}:  Validity or non-validity of \eqref{Rpgg} (left) and \eqref{RRpgg} (right).}
\label{figRpproof}
\includegraphics[width=0.8\textwidth]{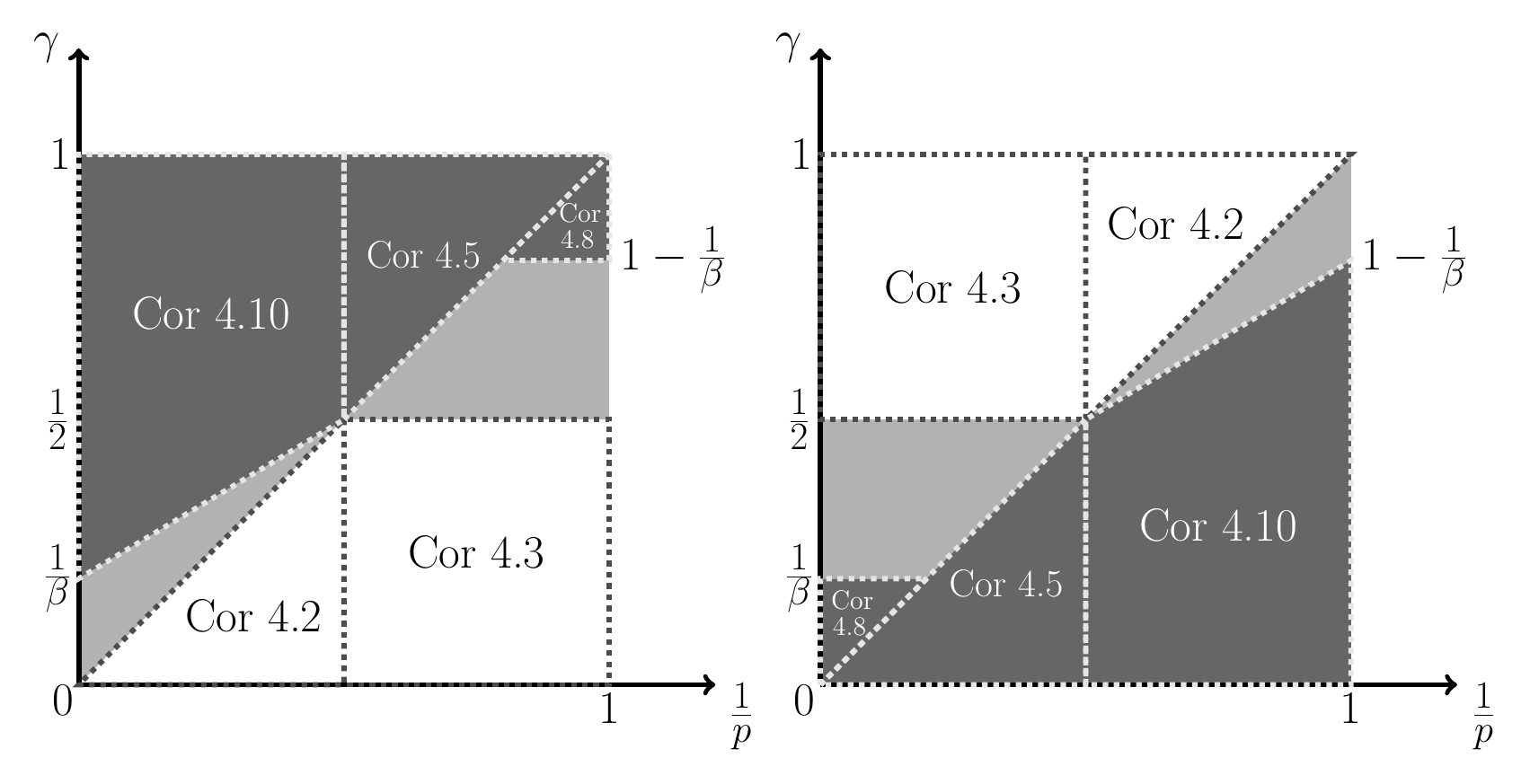}
\end{figure}

In the case of Vicsek cable systems (and in the case of Vicsek graphs), a more precise picture is available (see Theorem~\ref{ThDR}). However, the proofs in \cite{DR24} rely heavily on the specific structure of the Vicsek graphs, and a similar behavior cannot be expected to hold for all graphs with slow diffusion.

\medskip

For Riemannian manifolds, we obtain an analogous result:

\begin{theorem} \label{MainThRM}
Let \( (\Gamma, m) \) be a Riemannian manifold equipped with the (positive) Laplace–Beltrami operator \( \Delta \) and the Riemannian gradient \( \nabla \). Assume that \( \Gamma \) satisfies \eqref{LG}, \eqref{Dm}, and \eqref{UEm} for some \( \beta \geq 2 \). Then the following holds:
\begin{enumerate}[i.]
    \item \eqref{Rpgg} holds when \( p \in (1,\infty) \) and \( \gamma \leq \min\left\{\frac{1}{2}, \frac{1}{p} \right\} \),
    \item \eqref{RRpgg} holds when \( p \in (1,\infty) \) and \( \gamma \geq \max\left\{\frac{1}{2},  \frac{1}{p} \right\} \),
    \item \eqref{RRpgg} fails when \( p \in (1,\infty) \) and \( \gamma < \min\left\{ \max\left\{ \frac{1}{p}, \frac{1}{\beta} \right\}, \frac{1}{\beta} + \frac{1}{p} \left( 1 - \frac{2}{\beta} \right) \right\} \),
    \item \eqref{Rpgg} fails when \( p \in (1,\infty) \) and \( \gamma > \max\left\{ \min\left\{ \frac{1}{p}, 1 - \frac{1}{\beta} \right\}, \frac{1}{\beta} + \frac{1}{p} \left( 1 - \frac{2}{\beta} \right) \right\} \).
\end{enumerate}
In particular, when \( \beta > 2 \), \eqref{Rp} holds if and only if \( p \in (1,2]\) and \eqref{RRp} holds if and only if \( p \in (2,\infty] \).
\end{theorem}

\begin{figure}[!ht] 
\centering
\caption{Structure of the proof of Theorem~\ref{MainThRM}: Validity or non-validity of \eqref{Rpgg} (left) and \eqref{RRpgg} (right).}
\label{figRpproofm}
\includegraphics[width=0.8\textwidth]{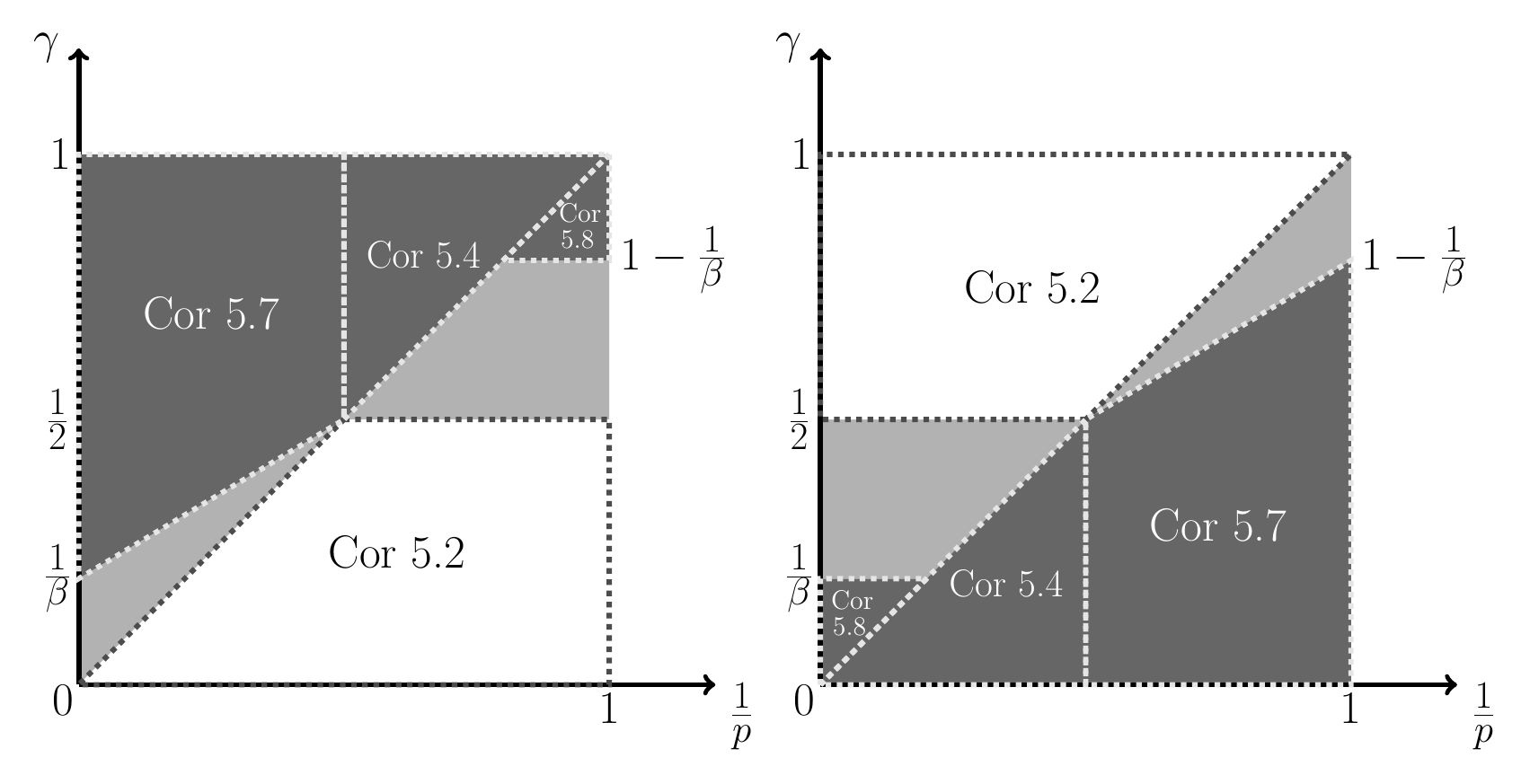}
\end{figure}

In Theorem~\ref{MainThRM}, the condition \eqref{LG} is not required for the proof of all cases. Specifically, it is only necessary for Corollaries \ref{Corg>1/p,1/2m} and \ref{Corg<1/p} (see Figure \ref{figRpproofm}). The condition \eqref{LG} is a local geometric condition on the Riemannian manifold ensuring the boundedness of \( \nabla e^{-\Delta} \) on \( L^p \) for all \( p \in [1,\infty] \). It is plausible that this assumption could be removed via a more refined argument, or by changing the definitions \eqref{Rpgm} and \eqref{RRpgm}. However, since our focus is on the non-local behavior of the manifold, we do not view the local assumption \eqref{LG} as a serious limitation in this work.

\medskip

We conclude this introduction with a list of conjectures that provide directions for future research.

\medskip

\noindent {\bf Conjectures.} All conjectures concern graphs equipped with an analytic Markov semigroup, and analogous ones can be stated for Riemannian manifolds.

\begin{enumerate}[A.]
\item (Coulhon, Duong) \eqref{Rp} holds for all $p\in (1,2]$.

\emph{Comment:} This conjecture is supported by examples of non-doubling Lie groups and tree graphs where \eqref{Rp} holds for all $p\in (1,2)$; see \cite{Sjo99, SV08, LMSTV??}. However, all existing proofs rely critically on either the group structure or a carefully chosen weight on the tree. Therefore, it remains unclear whether the doubling condition can be completely removed in general.

\smallskip

\item If $\Gamma$ is doubling and \eqref{Rp} holds for some $p \in (2, \infty)$, then the Markov kernel satisfies the Gaussian upper bound, i.e., \eqref{UE} with $\beta = 2$.

\emph{Comment:} This conjecture is a natural extension of the results in this article. However, proving it would require a new approach that does not rely on \emph{a priori} estimates on the heat kernel or the underlying Markov process.

\smallskip

\item (Devyver, Russ) On Vicsek graphs, the estimate \eqref{RRpgg} holds for all $p \in (2, \infty)$ and $\gamma > \frac{1}{\beta} + \frac{1}{p}\left(2 - \frac{2}{\beta}\right)$.

\emph{Comment:} One can further extend this conjecture by proposing that, on Vicsek graphs, \eqref{Rpgg} holds for all $p \in (1, 2)$ and $\gamma < \frac{1}{\beta} + \frac{1}{p}\left(2 - \frac{2}{\beta}\right)$. When $\gamma = \frac{1}{\beta} + \frac{1}{p}\left(2 - \frac{2}{\beta}\right)$, the validity of \eqref{Rpgg} and \eqref{RRpgg} remains an open question, and no conjecture has yet been proposed.

\smallskip

\item Let $\Gamma$ be a graph satisfying $V(x,r) \approx r^D$ and \eqref{UE}. If \eqref{RRpgg} holds for all $p \in (1, 2)$ and $\gamma > \frac{1}{\beta} + \frac{1}{p}\left(2 - \frac{2}{\beta}\right)$, then either $\beta = 2$ or the graph $\Gamma$ is a Vicsek graph, meaning that $D = \beta - 1$.

\emph{Comment:} This final conjecture is our own. It reflects our belief that the behavior observed in Vicsek graphs - illustrated in Figure \ref{figRpV} - cannot be generalized to non-Vicsek graphs.

\end{enumerate}

\medskip

\noindent {\bf Notation.} We write $A \lesssim B$ to indicate that there exists a constant $C > 0$, independent of the relevant parameters, such that $A \leq C B$. Similarly, we write $A \approx B$ when both $A \lesssim B$ and $B \lesssim A$.

\subsection*{Acknowledgements}
The author expresses gratitude to Emmanuel Russ for the insightful discussions on this topic, particularly for presenting his articles \cite{DRY21,DR24}, and for his encouragement in pursuing this work. He also warmly thanks Meng Yang for pointing out a mistake in an earlier version of this preprint.

\section{Definitions in graphs and Riemannian manifolds}

\label{Sspaces}

\subsection{Definitions of graphs}

A (weighted, unoriented) graph \((\Gamma,\mu)\) is the pairing of a countable (infinite) set \(\Gamma\) with a symmetric, nonnegative weight function \(\mu\) on \(\Gamma \times \Gamma\). The set of edges is defined by
\[
E := \{(x,y)\in \Gamma^2 \,:\, \mu_{xy} > 0\},
\]
and we write \(x \sim y\) when \((x,y) \in E\) (that is, \(x\) and \(y\) are neighbors). The graph is naturally equipped with a measure, a distance, and a random walk. The measure \(m\) on \(\Gamma\) is defined by
\[
m(x) = \sum_{x \sim y} \mu_{xy},
\]
and more generally, for a subset \(E \subset \Gamma\), we define
\[
m(E) := \sum_{x \in E} m(x).
\]
The distance \(d(x,y)\) between two vertices \(x,y \in \Gamma\) is the length of the shortest path connecting them, i.e., the smallest integer \(N\) such that there exists a sequence \(x = x_0, x_1, \dots, x_N = y\) with \(x_{i-1} \sim x_i\) for all \(i = 1,\dots,N\). We denote by \(B(x,r)\) the (discrete) ball
\[
B(x,r) := \{y \in \Gamma \,:\, d(x,y) \leq r\},
\]
and define \(V(x,r) := m(B(x,r))\). The weights \(\mu_{xy}\) define a discrete-time, reversible Markov kernel \(p\) via
\[
p(x,y) := \frac{\mu_{xy}}{m(x)m(y)}.
\]
The associated random walk is given by the discrete kernels \(p_k(x,y)\), defined recursively for all \(k \geq 0\) by
\begin{equation}
\left\{
\begin{array}{ll}
p_0(x,y) = \frac{\delta(x,y)}{m(y)}, \\
p_{k+1}(x,y) = \sum_{z \in \Gamma} p(x,z) \, p_k(z,y) \, m(z).
\end{array}
\right.
\end{equation}
For all \(k \geq 1\), the kernel \(p_k\) is symmetric, i.e., \(p_k(x,y) = p_k(y,x)\) for all \(x,y \in \Gamma\). Moreover, \(p_k(x,y)m(y)\) represents the probability of moving from \(x\) to \(y\) in \(k\) steps. The Markov operator \(P\) is defined via the kernel \(p\) with respect to the measure \(m\): for any function \(f\) and any \(x \in \Gamma\),
\[
Pf(x) = \sum_{y \in \Gamma} p(x,y) f(y) m(y).
\]
It is straightforward to check that \(p_k\) is the kernel of \(P^k\), and that \(P^k\) is a contraction on \(L^p(\Gamma)\) for all \(k \in \mathbb{N}\) and all \(p \in [1,\infty]\).

\medskip

\noindent The following assumption on the graph will always be made:

\begin{itemize}
    \item \textbf{Local uniform finiteness:} There exists \(M \in \mathbb{N}\) such that, for all \(x \in \Gamma\),
    \[
    \#B(x,1) = \#\left(\{y \in \Gamma \,:\, y \sim x\} \cup \{x\} \right) \leq M.
    \]
\end{itemize}
Moreover, the following additional conditions will be assumed when required:
\begin{itemize}
    \item \textbf{Non-degenerate probability transitions:} The Markov process satisfies
    \begin{equation} \label{defLBE}
    p(x,y)m(y) \geq \varepsilon \qquad \text{for all } x,y \in \Gamma \text{ with } x \sim y,\, x \neq y.
    \end{equation}
    \item \textbf{Analyticity of the semigroup:} The Markov semigroup \((P^k)_{k \in \mathbb{N}}\) is said to be \emph{analytic} if there exists \(C > 0\) such that
    \[
    \|(I-P)P^k f\|_{L^2(\Gamma)} \leq \frac{C}{k+1} \|f\|_{L^2(\Gamma)} \qquad \text{for all } k \in \mathbb{N},\, f \in L^2(\Gamma).
    \]
    
    \item \textbf{Doubling condition:} The graph \((\Gamma,\mu)\) is said to be \emph{doubling}- or to satisfy \eqref{D} - if there exists \(C > 0\) such that
    \begin{equation} \label{D} \tag{D}
    V(x,2r) \leq C V(x,r) \qquad \text{for all } x \in \Gamma,\ r \in \mathbb{N}.
    \end{equation}
    
    \item \textbf{Sub-Gaussian estimates:} The Markov kernel satisfies pointwise sub-Gaussian upper bounds with escape rate \(r^\beta\), i.e., there exist constants \(C,c > 0\) such that, for all \(k \in \mathbb{N}^*\) and all \(x,y \in \Gamma\),
    \begin{equation} \label{UE} \tag{UE$_\beta$}
    p_k(x,y) \leq \frac{C}{V(x,k^{1/\beta})} \exp\left(-c \left[\frac{d(x,y)^\beta}{k}\right]^{\frac{1}{\beta-1}}\right).
    \end{equation}
\end{itemize}

\medskip

Some comments are in order. The analyticity of Markov operators - the discrete analogue of semigroup analyticity - has been studied in, e.g., \cite{CSC90,Dun06,LM14,FenLB}. For reversible Markov processes, the following characterization holds:

\begin{proposition}[\cite{FenLB}, Theorem 1.7] \label{PrLB}
Let \((\Gamma,\mu)\) be a locally uniformly finite and doubling graph. Then the following are equivalent:
\begin{enumerate}[(i)]
    \item \((P^k)_{k \in \mathbb{N}}\) is analytic in \(L^2\);
    \item \((P^k)_{k \in \mathbb{N}}\) is analytic in \(L^q\) for all \(q \in (1,\infty)\), i.e., there exists \(C_q > 0\) such that
    \[
    \|(I-P)P^k f\|_{L^q(\Gamma)} \leq \frac{C_q}{k+1} \|f\|_{L^q(\Gamma)} \qquad \text{for all } k \in \mathbb{N},\, f \in L^q(\Gamma);
    \]
    \item \(-1\) is not in the \(L^2\)-spectrum of \(P\);
    \item There exist \(\ell \in \mathbb{N}\), \(\varepsilon > 0\) such that
    \[
    p_{2\ell+1}(x,x)m(x) \geq \varepsilon \qquad \text{for all } x \in \Gamma.
    \]
\end{enumerate}
If, in addition, the Markov kernel has non-degenerate probability transitions, then the above are also equivalent to:
\begin{enumerate}[(i)] \addtocounter{enumi}{4}
    \item There exists \(\ell \in \mathbb{N}\) such that for every vertex \(x \in \Gamma\), there exists a closed path of length \(2\ell+1\) starting and ending at \(x\).
\end{enumerate}
\end{proposition}

\medskip

The doubling condition implies that the volume growth of balls is at most polynomial, i.e., there exist constants \(C,D > 0\) such that
\begin{equation} \label{defD}
V(x, \Lambda r) \leq C \Lambda^D V(x,r) \qquad \text{for all } x \in \Gamma,\ r \in \mathbb{N},\ \Lambda \geq 1.
\end{equation}
Conversely, since the graph is unbounded and connected, its volume growth is at least polynomial: there exist \(C,\delta > 0\) such that
\begin{equation} \label{defdelta}
V(x, \Lambda r) \geq C^{-1} \Lambda^\delta V(x,r) \qquad \text{for all } x \in \Gamma,\ r \in \mathbb{N},\ \Lambda \geq 1.
\end{equation}

The parameter \(\beta\) in \eqref{UE} corresponds to the escape rate of the diffusion: the average time needed for a particle starting at \(x\) to escape the ball \(B(x,r)\) scales like \(k = r^\beta\). Barlow showed in \cite{Bar04} that one can construct a graph \((\Gamma^{D,\beta}, \mu^{D,\beta})\) satisfying \(V(x,r) \approx r^D\) and \eqref{UE} if and only if \(D \in (1,\infty)\) and \(\beta \in [2, D+1]\). The case \(\beta = D+1\) corresponds to the Vicsek graphs (see Theorem~\ref{ThCCFR} and Figure~\ref{fig1}). The case \(D = \log_2(3)\), \(\beta = \log_2(5)\) corresponds to the graphical Sierpinski gasket, shown in Figure~\ref{Sierpinski6}.

\begin{figure}[!ht]
\centering
\caption{Graphical Sierpinski gasket (6th step of the construction)}
\label{Sierpinski6}
\includegraphics[width=0.45\textwidth]{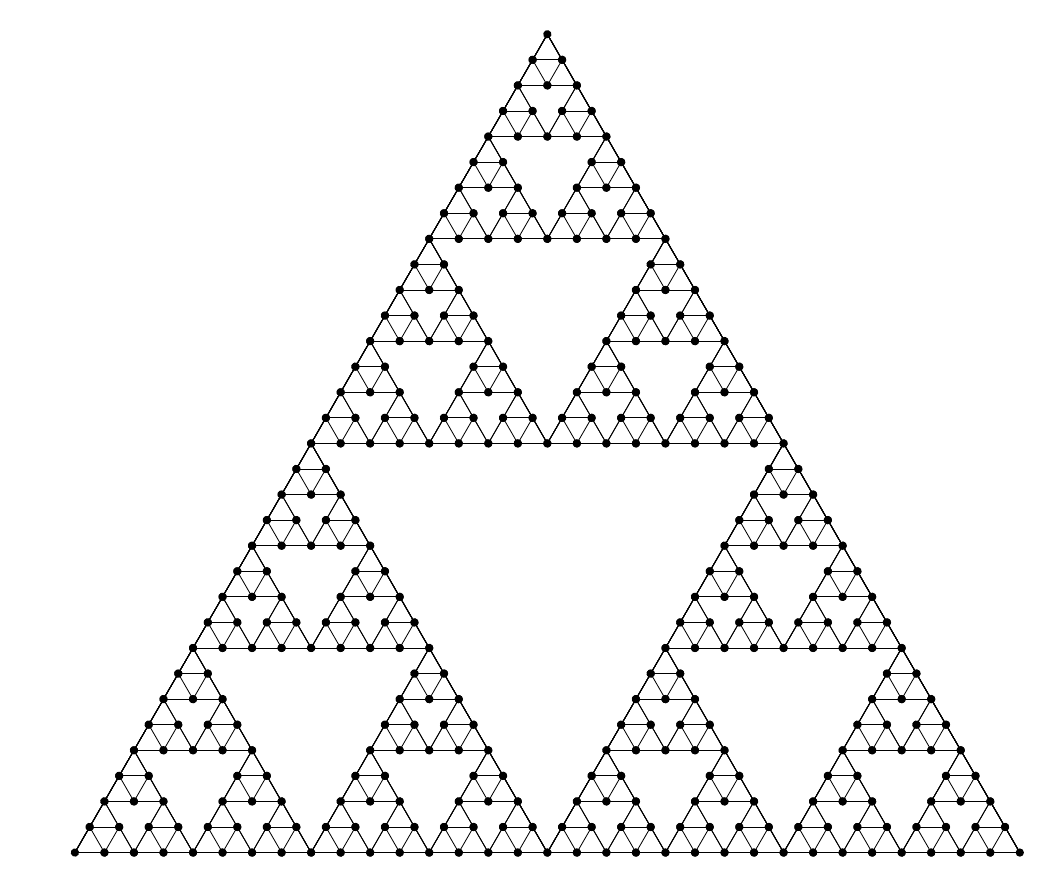}
\end{figure}

\medskip

Combining analyticity and sub-Gaussian estimates, one obtains (see \cite[Theorem 1.1]{Dun06}) the following bound on the discrete-time derivative of $p_k$:
\begin{equation} \label{dkUE}
|p_k(x,y) - p_{k-1}(x,y)| \leq \frac{C}{k V(x,k^{1/\beta})} \exp\left(- c \left[\frac{d(x,y)^\beta}{k} \right]^{\frac{1}{\beta - 1}}\right),
\end{equation}
for all \(k \in \mathbb{N}^*\), \(x,y \in \Gamma\), with constants \(C,c > 0\).

\medskip

The Laplacian is defined as \(\Delta := I - P\), and the gradient length by
\[
\nabla f(x) := \left( \frac{1}{2} \sum_{y \sim x} p(x,y) |f(y) - f(x)|^2 m(y) \right)^{1/2}.
\]
We have the isometry
\begin{equation} \label{E2} \tag{E\(_2\)}
\|\nabla f\|_{L^2}^2 = \sum_{x \in \Gamma} (I - P)[f](x) f(x) = \|\Delta^{1/2} f\|_{L^2}^2 \qquad \text{ for } f\in L^2(\Gamma).
\end{equation}

Finally, we say that \((\Gamma,\mu)\) satisfies the Poincaré inequality \eqref{Ppa} if there exists \(C > 0\) such that for any ball \(B := B(x,r)\) and any function \(f\) supported in \(B\),
\begin{equation} \label{Ppa} \tag{P\(_{p,\alpha}\)}
\|f\|_{L^p(B)} \leq C r^\alpha \|\nabla f\|_{L^p(B)}.
\end{equation}
The case $p=2$, $\alpha =1$ is known as the Faber–Krahn inequality.

\subsection{Definitions in Riemannian manifolds} \label{Ssmanifold}
When referring to a Riemannian manifold $(\Gamma, m)$, we mean a smooth, complete, connected, and non-compact Riemannian manifold equipped with the geodesic distance $d$ and the Riemannian measure $m$. We denote by $B(x,r)$ the open ball centered at $x \in \Gamma$ with geodesic radius $r > 0$, i.e.,
\[
B(x,r) := \{y \in \Gamma : d(x,y) < r\}.
\]
As in the graph setting, we write $V(x,r) := m(B(x,r))$ and say that the Riemannian manifold is \emph{doubling} if there exists a constant $C > 0$ such that
\begin{equation} \label{Dm} \tag{D}
V(x,2r) \leq C V(x,r) \qquad \text{ for all } x \in \Gamma,\; r > 0.
\end{equation}
The doubling property implies that the volume of balls grows at most and at least polynomially. That is, there exist constants $C > 0$, $D > 0$, and $\delta > 0$ such that
\begin{equation} \label{defDm}
V(x, \Lambda r) \leq C \Lambda^D V(x,r) \qquad \text{ for all } x \in \Gamma,\; r \in \mathbb{N},\; \Lambda \geq 1,
\end{equation}
and, since $\Gamma$ is connected and unbounded,
\begin{equation} \label{defdeltam}
V(x, \Lambda r) \geq C^{-1} \Lambda^\delta V(x,r) \qquad \text{ for all } x \in \Gamma,\; r \in \mathbb{N},\; \Lambda \geq 1.
\end{equation}

We denote by $\nabla$ the Riemannian gradient and by $\Delta$ the (non-negative) Laplace–Beltrami operator on $\Gamma$. The heat semigroup $(e^{-t\Delta})_{t>0}$ is generated by $-\Delta$, and we denote its kernel by $h_t(x,y)$. It is well known (see, e.g., \cite[Theorem 7.13]{Gri}) that $h_t(x,y)$ is everywhere positive, symmetric in $x$ and $y$, and smooth in $(t,x,y) \in (0,\infty) \times \Gamma \times \Gamma$. 
Moreover, the heat semigroup is analytic on $L^p$ for all $p \in (1, \infty)$, meaning that for any $\gamma > 0$, there exists $C_{p,\gamma} > 0$ such that
\begin{equation} \label{analyt}
\|\Delta^{\gamma} e^{-t\Delta}\|_{L^p \to L^p} \leq \frac{C_{p,\gamma}}{t^{\gamma}}.
\end{equation}
For $\beta \geq 2$, we say that $\Gamma$ satisfies \eqref{UEm} if there exist constants $C, c > 0$ such that
\def\arraystretch{2.5}
\begin{equation} \label{UEm} \tag{UE$_\beta$}
h_t(x,y) \leq C \left\{ 
\begin{array}{ll}
\exp\left(- c \dfrac{d^2(x,y)}{t} \right) & \text{for } t \in (0,1], \; x, y \in \Gamma, \\
\exp\left(- c \left[\dfrac{d^\beta(x,y)}{t} \right]^{1/(\beta - 1)} \right) & \text{for } t > 1, \; x, y \in \Gamma.
\end{array} \right.
\end{equation}
By replacing vertices by spheres and edges by cylinder, we can construct Riemannian manifolds that have the same non-local diffusive behavior as graphs, see for instance \cite[Appendix]{CCFR15}.  Using this transformation, a graph satisfying \eqref{D} and \eqref{UE} will become a Riemannian manifold that satisfies \eqref{Dm} and \eqref{UE}. A Vicsek manifold is a Riemannian manifold obtained from a Vicsek graph using the presented transformation.

\medskip

By combining the pointwise estimate \eqref{UEm} with the analyticity, we obtain
\begin{equation} \label{dtUE}
|\partial_t h_t(x,y)| \leq \frac{C}{t} \left\{
\begin{array}{ll}
\exp\left(- c \dfrac{d^2(x,y)}{t} \right) & \text{for } t \in (0,1], \; x, y \in \Gamma, \\
\exp\left(- c \left[\dfrac{d^\beta(x,y)}{t} \right]^{1/(\beta - 1)} \right) & \text{for } t > 1, \; x, y \in \Gamma,
\end{array} \right.
\end{equation}
with constants $C, c > 0$ independent of $t > 0$, $x, y \in \Gamma$.

Since $\Delta$ is self-adjoint, spectral theory applies to fractional powers of $\Delta$ on a dense subspace of $L^2$, and thus also on a dense subspace of $L^p$ for $p \in [1, \infty)$. From the definitions of the Riemannian gradient and the Laplace–Beltrami operator, we have the identity
\begin{equation} \label{E2m} \tag{E$_2$}
\|\nabla f\|_{L^2(\Gamma)} = \langle \Delta f, f \rangle = \|\Delta^{1/2} f\|_{L^2(\Gamma)} \qquad \text{for all } f \in C_0^\infty(\Gamma).
\end{equation}

We say that $(\Gamma, m)$ satisfies the Poincaré inequality \eqref{Ppam} if, for any $\epsilon > 0$, there exists a constant $C_\epsilon > 0$ such that for any ball $B := B(x,r) \subset \Gamma$ with radius $r \geq 1$ and any function $f \in C^\infty_0(B)$ satisfying $\|e^{-t\Delta} f\|_p \geq \epsilon \|f\|_p$, we have
\begin{equation} \label{Ppam} \tag{P$_{p,\alpha}$}
\|e^{-\Delta} f\|_{L^p(B)} \leq C_\epsilon r^{\alpha} \|\nabla f\|_{L^p(B)}.
\end{equation}
The above Poincar\'e inequality will be an intermediate tool in our proof, and the presence of $e^{-\Delta}$ in \eqref{Ppam} is directly linked to the presence of $e^{-\Delta}$ in \eqref{RRpgm}.

\medskip

We also assume a local pointwise bound on the gradient of the heat kernel. Specifically, we say that $\Gamma$ satisfies \eqref{LG} if there exist constants $C, c > 0$ such that
\begin{equation} \label{LG} \tag{LG}
|\nabla_x h_{1}(x,y)| \leq \frac{C}{V(x,1)} \exp(- c d(x,y)^2) \qquad \text{for all } x, y \in \Gamma.
\end{equation}
This estimate on the gradient of the heat kernel holds if the Riemannian manifold is well approximated by Euclidean planes in balls of radius 1. Indeed, Theorem 1.1 in \cite{DRY21}\footnote{More precisely, a local version of that theorem, which can be established using the same proof.} shows that $\Gamma$ satisfies \eqref{LG} whenever it satisfies \eqref{Dm} and \eqref{UEm}, along with the following local assumptions:
\begin{itemize}
\item \emph{Local Ahlfors regularity of balls:} There exists $d \in \mathbb{N}$ such that $V(x,r) \approx r^d$ for all $x \in \Gamma$ and $r \in (0,1)$.
\item \emph{Local reverse Hölder estimate for gradients:} There exists $C > 0$ such that for every ball $B = B(x,r)$ with $r < 1$ and every function $u$ satisfying $\Delta u = 0$ in $2B$, we have
\begin{equation} \label{LRH} \tag{LRH}
\sup_{B} |\nabla u| \leq C \fint_{2B} |u| \, dm.
\end{equation}
\end{itemize}

\section{Basic tools}
In this section, either \( \Gamma = (\Gamma, \mu) \) is a (weighted, unoriented) graph that is locally uniformly finite and whose Markov semigroup is analytic, or \( \Gamma = (\Gamma, m) \) is a Riemannian manifold equipped with the positive Laplace-Beltrami operator \( \Delta \) and the Riemannian gradient \( \nabla \).

\begin{theorem}[Interpolation] \label{ThInterpolation}
Let \( p, q, r \in [1,\infty] \) and \( \gamma, \delta, \epsilon \in [0,1] \) be such that there exists \( \theta \in (0,1) \) satisfying 
\[
\frac{1}{r} = \frac{\theta}{p} + \frac{1-\theta}{q} \quad \text{and} \quad \epsilon = \theta \gamma + (1-\theta) \delta.
\]
Then
\[
\text{(R}_{p,\gamma}\text{)} + \text{(R}_{q,\delta}\text{)} \implies \text{(R}_{r,\epsilon}\text{)},
\]
\[
\text{(RR}_{p,\gamma}\text{)} + \text{(RR}_{q,\delta}\text{)} \implies \text{(RR}_{r,\epsilon}\text{)},
\]
and if \( \beta \geq 2 \),
\[
\text{(P}_{p,\beta \gamma}\text{)} + \text{(P}_{q,\beta \delta}\text{)} \implies \text{(P}_{r,\beta \epsilon}\text{)}.
\]
\end{theorem}

\begin{proof}
The theorem is a straightforward consequence of complex interpolation of \( L^p \) spaces. Let us add a few remarks to clarify this:
\begin{itemize}
    \item When \( \gamma \in (0,1) \) and \( p \in (1,\infty) \), the norm \( \|\Delta^{\gamma} f\|_p \) is equivalent to the \( L^p \)-norm of a Littlewood-Paley-Stein functional. Specifically: if \( \Gamma \) is a Riemannian manifold,
        \[
        \|\Delta^{\gamma} f\|_p \approx \left\| \left( \int_0^\infty t^{1-2\gamma} |\Delta e^{-t\Delta} f|^2 \, dt \right)^{1/2} \right\|_p,
        \]
        and if \( \Gamma \) is a graph,
        \[
        \|\Delta^{\gamma} f\|_p \approx \left\| \left( \sum_{k=0}^\infty k^{1-2\gamma} |\Delta P^k f|^2 \right)^{1/2} \right\|_p.
        \]
        See for instance \cite{Stein} for the equivalence between the $L^p$ norm and the $L^p$ norm of continuous Littlewood-Paley-Stein functionals, and \cite{LM14,FenLPS} for the analogue with discrete time Littlewood-Paley-Stein functionals.

    \item When \( \Gamma \) is a graph, the quantity \( \nabla f(x) \) is itself an \( L^2 \)-norm over the edges adjacent to \( x \), so \( \|\nabla f\|_p \) can be interpreted as an \( L^pL^2 \)-norm on the edges of the graph.
\end{itemize}

These observations indicate that the interpolation results in Theorem~\ref{ThInterpolation} are indeed consequences of complex interpolation of sublinear operators in \( L^p \).
\end{proof}

\begin{theorem}[Duality] \label{ThDuality}
Let \( p \in (1,\infty) \) and \( \gamma \in [0,1] \). Then
\[
\text{(R}_{p,\gamma}\text{)} \implies \text{(RR}_{p',1-\gamma}\text{)},
\]
where \( p' \) denotes the Hölder conjugate of \( p \).
\end{theorem}

\begin{proof}
This result is also classical. In the case of Riemannian manifolds, the proof follows the same reasoning as Lemma 1.4 in \cite{DR24}. The case of graphs is an even simpler variant.
\end{proof}

\section{Proof of Theorem \ref{MainTh}} \label{Sproofg}
\subsection{First results}

\begin{lemma} \label{lemRp1} 
Let $(\Gamma,\mu)$ be a graph with the local doubling property, i.e., there exists a constant $C > 0$ such that $V(x,1) \leq C m(x)$ for all $x \in \Gamma$. Then the inequalities (R$_{p,0}$) and (RR$_{p,1}$) hold for all $p \in [1,\infty]$. More precisely, for all $f \in L^p(\Gamma)$ and $x \in \Gamma$, we have
\begin{equation} \label{Rp0a}
|\nabla f(x)| \leq \frac{C}{V(x,1)} \sum_{y \in B(x,1)} |f(y)|\, m(y),
\end{equation}
where $C>0$ depends only on the local doubling constant, and
\begin{equation} \label{Rp1a}
|\Delta f(x)| \leq \sqrt{2} \nabla f(x).
\end{equation}
\end{lemma}

\begin{proof}
For any $x \in \Gamma$ and any function $f$ on $\Gamma$, we have
\begin{align*}
|\nabla f(x)| 
&= \left( \frac{1}{2} \sum_{y \sim x} p(x,y) |f(x) - f(y)|^2\, m(y) \right)^{1/2} \\
&\leq \frac{1}{\sqrt{2}} |f(x)| + \left( \frac{1}{2} \sum_{y \sim x} |f(y)|^2 \right)^{1/2} \\
&\lesssim \sum_{y \in B(x,1)} |f(y)| \lesssim \frac{1}{V(x,1)} \sum_{y \in B(x,1)} |f(y)|\, m(y),
\end{align*}
where the last inequality follows from the local doubling property. We then estimate the $L^p$ norm:
\begin{align*}
\|\nabla f\|_p^p 
&\lesssim \sum_{x \in \Gamma} \left( \frac{1}{V(x,1)} \sum_{y \in B(x,1)} |f(y)|\, m(y) \right)^p m(x) \\
&\leq \sum_{x \in \Gamma} \frac{1}{V(x,1)} \sum_{y \in B(x,1)} |f(y)|^p\, m(y)\, m(x) \\
&= \sum_{y \in \Gamma} |f(y)|^p\, m(y) \sum_{x \in B(y,1)} \frac{1}{V(x,1)} \lesssim \|f\|_p^p,
\end{align*}
where we again used the local doubling property:
\[
\sum_{x \in B(y,1)} \frac{1}{V(x,1)} \leq \sum_{x \in B(y,1)} \frac{1}{m(y)} \lesssim \sum_{x \in B(y,1)} \frac{1}{V(y,1)} = 1.
\]
To prove \eqref{Rp1a}, observe:
\begin{align*}
|\Delta f(x)| 
&= \left| \sum_{y \sim x} p(x,y) [f(x) - f(y)]\, m(y) \right| \\
&\leq \sum_{y \sim x} p(x,y) |f(x) - f(y)|\, m(y) \\
&\leq \left( \sum_{y \sim x} p(x,y) |f(x) - f(y)|^2\, m(y) \right)^{1/2} = \sqrt{2} \nabla f(x).
\end{align*}
This concludes the proof.
\end{proof}

\begin{corollary} \label{Corg>1/p}
Let $(\Gamma,\mu)$ be a graph with the local doubling property. Then the bound \eqref{Rpgg} holds for all $p \in [2,\infty)$ and $\gamma \in [0,\frac{1}{p}]$, and the bound (RR$_{p,\gamma}$) holds for all $p \in (1,2]$ and $\gamma \in [\frac{1}{p},1]$.
\end{corollary}

\begin{proof}
Lemma \ref{lemRp1} guarantees that (R$_{p,0}$) holds for all $p \in [1,\infty]$, while (R$_{2,1/2}$) holds by construction of the gradient in \eqref{E2}. The conclusion for \eqref{Rpgg} follows by interpolation via Theorem \ref{ThInterpolation}. Similarly, interpolating (RR$_{p,1}$) - established in Lemma \ref{lemRp1} - and (RR$_{2,1/2}$) yields the claimed range for \eqref{RRpgg}.
\end{proof}

\begin{corollary} \label{Corg>1/2}
Let $(\Gamma,\mu)$ be an unbounded doubling graph such that the Markov semigroup is analytic and satisfies \eqref{UE}. Then \eqref{Rpgg} holds for all $p \in (1,2)$ and $\gamma \in [0,\frac{1}{2}]$, and \eqref{RRpgg} holds for all $p \in (2,\infty)$ and $\gamma \in [\frac{1}{2},1]$.
\end{corollary}

\begin{proof}
The main result of \cite{CCFR15} shows that \eqref{Rp} holds for all $p \in (1,2]$. Lemma \ref{lemRp1} provides (R$_{p,0}$) for all $p\in [1,\infty]$. The interpolation (Theorem \ref{ThInterpolation}) implies that \eqref{Rpgg} holds whenever $p \in (1,2)$ and $\gamma \leq \frac{1}{2}$. Then, by Theorem \ref{ThDuality}, \eqref{RRpgg} is true whenever $p \in (2,\infty)$ and $\gamma \geq \frac{1}{2}$.
\end{proof}

\begin{lemma} \label{lemp=2}
Let $(\Gamma,\mu)$ be an unbounded doubling graph. Then (RR$_{2,\gamma}$) holds if and only if $\gamma \geq \frac{1}{2}$.
\end{lemma}

\begin{proof}
The fact that (RR$_{2,\gamma}$) holds for $\gamma \geq \frac{1}{2}$ follows by interpolating between \eqref{E2} and (RR$_{2,1}$), the latter given by Lemma \ref{lemRp1}.

To show failure when $\gamma < \frac{1}{2}$, we note that $0$ lies in the spectrum of $\Delta$. Since the domain is unbounded, we can construct functions $f_k \in L^2(\Gamma)$ as
\[
f_k(x) := \max\left\{0, 1 - \frac{d(x_0,x)}{k}\right\},
\]
where $x_0 \in \Gamma$ is fixed. These functions approximate the constant function $\mathbf{1}_\Gamma$ in the sense that
\[
|(I - P)f_k(x)| \leq \frac{1}{k} \mathbf{1}_{B(x_0,k+1)}(x),
\]
so
\[
\|\Delta f_k\|_2 = \|(I - P)f_k\|_2 \leq \frac{C}{k} \|f_k\|_2,
\]
for some $C > 0$ depending only on the doubling constant.
Thus, $0 \in \text{Sp}(\Delta)$. Let $dE(\lambda)$ denote the spectral measure. Then
\[
f_k = \int_0^{C/k} dE_{f_k}(\lambda), \quad \text{so} \quad \|\Delta^\alpha f_k\|_2^2 = \int_0^{C/k} \lambda^{2\alpha} \, dE_{f_k,f_k}(\lambda).
\]
Hence, for any $\gamma < \frac{1}{2}$,
\begin{align*}
\|\nabla f_k\|_2^2 &= \|\Delta^{1/2} f_k\|_2^2 = \int_0^{C/k} \lambda \, dE_{f_k,f_k}(\lambda) \\
&\leq (C/k)^{1 - 2\gamma} \int_0^{C/k} \lambda^{2\gamma} \, dE_{f_k,f_k}(\lambda) = (C/k)^{1 - 2\gamma} \|\Delta^\gamma f_k\|_2^2,
\end{align*}
which shows that (RR$_{2,\gamma}$) fails.
\end{proof}

\begin{corollary} \label{Corg<1/p}
Let $(\Gamma,\mu)$ be an unbounded doubling graph such that the Markov semigroup is analytic and satisfies \eqref{UE}. Then \eqref{RRpgg} fails for all $p \in (2,\infty)$ and $\gamma < \frac{1}{p}$.
\end{corollary}

\begin{proof}
Assume for contradiction that (RR$_{p,\gamma}$) holds for some $p \in (2,\infty)$ and $\gamma < \frac{1}{p}$. Then interpolation with (RR$_{1,1}$) - provided by Lemma \ref{lemRp1} - gives that (RR$_{2,\delta}$) holds for some $\delta < \frac{1}{2}$, contradicting Lemma \ref{lemp=2}.
\end{proof}

\subsection{Further results} 
\label{SsFG}

\begin{proposition} \label{prUE=>aFK}
Let $(\Gamma,\mu)$ be an unbounded doubling graph such that the Markov semigroup is analytic and satisfies \eqref{UE}. Then for any \( p \in (1,\infty) \) and any \( \alpha \geq 0 \), there exists a constant \( C = C(\alpha, p) > 0 \) such that for any ball \( B := B(x, r) \subset \Gamma \) and any function \( f \in L^p(\Gamma, m) \) supported in \( B \), we have
\begin{equation} \label{UE=>aFK}
\|f\|_{L^p} \leq C r^{\alpha\beta} \|\Delta^{\alpha} f\|_{L^p}.
\end{equation}
\end{proposition}

\begin{remark} \label{RmkRR=>P}
Thanks to Proposition \ref{prUE=>aFK}, we obtain the implication
\[
\text{ \eqref{RRpgg}} \implies \text{ (P$_{p,\beta \gamma}$) }.
\]
\end{remark}

\bp 
We begin with the case \( \alpha \in [0,1) \), which is simpler and will be the only one needed later in our arguments. The more technical case \( \alpha \geq 1 \) will be treated afterward.

\medskip

\noindent We may assume without loss of generality that \( f \not\equiv 0 \), since the result is trivial otherwise. Set \( k = \lambda^\beta r^\beta \) for some \( \lambda \geq 1 \) to be chosen later. Then we decompose \( f \) as
\begin{equation} \label{UE->FK0}
f = P^{k+1} f + \sum_{\ell = 0}^k \Delta P^\ell f.
\end{equation}

\medskip

\noindent \textbf{Step 1.} Estimating the term \( P^{k+1} f \). We claim that for a suitable choice of \( \lambda \), we have
\begin{equation} \label{claim1}
\|P^{k+1} f\|_p \leq \frac{1}{2} \|f\|_p.
\end{equation}
Indeed,
\begin{multline} \label{UE->FK2}
\|P^{k+1} f\|_{L^p} := \left( \sum_{x\in \Gamma } \left| \sum_{y\in \Gamma} p_{k+1}(x,y) f(y) m(y) \right| ^p m(x)\right)^{\frac{1}{p}} \\
\leq \left( \sum_{x\in \Gamma} \sup_{y\in B} |p_{k+1}(x,y)|^p m(x) \right)^{\frac{1}{p}} \|f\|_{L^1}.
\end{multline}
Using the doubling property and Hölder's inequality, we get
\begin{equation} \label{UE->FK1}
\|f\|_{L^1} \leq m(B)^{1 - \frac{1}{p}} \|f\|_p.
\end{equation}
Let \( C_0(\lambda B) := 2\lambda B \), and for \( j \geq 1 \), let \( C_j(\lambda B) := 2^{j+1}\lambda B \setminus 2^{j}\lambda B \). For \( x \in C_j(\lambda B) \) and \( y \in B \), we have \( d(x, y)^\beta \gtrsim (2^{j\beta} - 1)(k+1) \), hence by \eqref{UE},
\[
\sup_{y \in B} |p_{k+1}(x,y)| \lesssim \frac{1}{m(\lambda B)} e^{-c2^{j\beta/(\beta-1)}}.
\]
Thus,
\begin{multline*}
\sum_{x \in \Gamma} \sup_{y \in B} |p_{k+1}(x,y)|^p m(x) \lesssim \sum_{j \in \mathbb{N}} m(2^{j+1} \lambda B) \left( \frac{e^{-c2^{j\beta/(\beta-1)}}}{m(\lambda B)} \right)^p \\
\lesssim \sum_{j \in \mathbb{N}} 2^{(j+1)D} m(\lambda B) \left( \frac{e^{-c2^{j\beta/(\beta-1)}}}{m(\lambda B)} \right)^p 
\lesssim \frac{1}{m(\lambda B)^{p-1}} \lesssim \frac{\lambda^{-(p-1)\delta}}{m(B)^{p-1}}.
\end{multline*}
Choosing \( \lambda \) sufficiently large (independently of \( x \) or \( r \)), we get
\[
\left( \sum_{x \in \Gamma} \sup_{y \in B} |p_{k+1}(x,y)|^p m(x) \right)^{\frac{1}{p}} \leq \frac{1}{2} m(B)^{\frac{1}{p} - 1} \frac{\|f\|_p}{\|f\|_1} \leq \frac{1}{2}  \frac{\|f\|_p}{\|f\|_1}.
\]
Injecting this into \eqref{UE->FK2}, we obtain \eqref{claim1}.

\medskip

\noindent \textbf{Step 2.} The remainder term in \eqref{UE->FK0} can be rewritten as
\[
\sum_{\ell = 0}^k \Delta P^\ell f = \sum_{\ell = 0}^{k} \Delta^{1-\alpha} P^\ell [\Delta^\alpha f].
\]
Since \( (P^\ell)_{\ell \in \mathbb{N}} \) is an analytic Markov semigroup on \( L^p \), we apply \cite[Proposition 2]{CSC90} to get
\[
\|\Delta^{1-\alpha} P^\ell g\|_{L^p} \leq C (\ell+1)^{\alpha - 1} \|g\|_{L^p}, \quad \forall g \in L^p.
\]
Hence,
\[
\left\| \sum_{\ell = 0}^k \Delta P^\ell f \right\|_{L^p} \lesssim \sum_{\ell = 0}^k (\ell+1)^{\alpha - 1} \|\Delta^\alpha f\|_{L^p} \lesssim (k+1)^{\alpha} \|\Delta^\alpha f\|_{L^p} \lesssim r^{\alpha\beta} \|\Delta^\alpha f\|_{L^p}.
\]

\medskip

\noindent \textbf{Step 3.} Combining the estimates from Steps 1 and 2 into \eqref{UE->FK0}, we get
\[
\|f\|_{L^p(\Gamma)} \leq \frac{1}{2} \|f\|_{L^p(\Gamma)} + C r^{\alpha\beta} \|\Delta^\alpha f\|_{L^p(\Gamma)}.
\]
By moving all the $\|f\|_p$ on the left-hand size, we get the desired estimate \eqref{UE=>aFK} for \( \alpha \in [0,1) \).

\medskip

\noindent \textbf{Step 4. Case \( \alpha \in [1,2) \).} The strategy is the same, but we need a suitable identity like \eqref{UE->FK0}. Namely,
\begin{equation} \label{UE->FK00}
f = P^{k+1} f + \sum_{\ell = 0}^k \Delta P^\ell f = P^{k+1} f + \sum_{\ell = 0}^k \Delta P^\ell \left( P^{k+1} f + \sum_{n=0}^k \Delta P^n f \right).
\end{equation}
So we decompose
\[
f = \underbrace{P^{k+1} f + \sum_{\ell = 0}^k \Delta P^{\ell+k+1} f}_{I_1(f): \text{ treated as in Step 1}} + \underbrace{\sum_{\ell = 0}^k \sum_{n = 0}^k \Delta^2 P^{\ell+n} f}_{I_2(f): \text{ treated as in Step 2}}.
\]
Using the $L^p$ boundedness of \( P \), we have
\[
\|I_1(f)\|_{L^p} \lesssim \|P^{k+1} f\|_{L^p} \leq \frac{1}{2} \|f\|_{L^p}.
\]
if $\lambda$ is larege enough. For \( I_2(f) \), the analyticity yields
\[
\|I_2(f)\|_{L^p} \lesssim \sum_{\ell = 0}^k \sum_{n = 0}^k (\ell + n + 1)^{\alpha - 2} \|\Delta^\alpha f\|_{L^p} \lesssim (k+1)^\alpha \|\Delta^\alpha f\|_{L^p}.
\]
Thus,
\[
\|f\|_{L^p} \leq \frac{1}{2} \|f\|_{L^p} + C r^{\alpha\beta} \|\Delta^\alpha f\|_{L^p},
\]
establishing \eqref{UE=>aFK} when \( \alpha \in [1,2) \).

\medskip

\noindent \textbf{Step 5. General case \( \alpha > 0 \).} Let \( N \in \mathbb{N} \) be the smallest integer strictly greater than \( \alpha \). One can iterate the decomposition in \eqref{UE->FK00} \( N \) times to obtain an identity where every term is either a linear combination of \( P^{k+1} f \) or \( \Delta^N f \). A similar reasoning as above completes the proof.
\ep

\begin{corollary} \label{Corg<1/b}
Let $(\Gamma,\mu)$ be an unbounded doubling graph such that the Markov semigroup is analytic and satisfies \eqref{UE}. Then \eqref{RRpgg} fails whenever $p \in (1,\infty)$ and $\gamma \in \left[0, \frac{1}{\beta}\right)$.
\end{corollary}

\bp
We argue by contraposition. Suppose that \eqref{RRpgg} holds, and we aim to prove that $\gamma \geq \frac{1}{\beta}$. By Remark \ref{RmkRR=>P}, this implies that (P$_{\beta,\gamma}$) holds, i.e., there exists $C > 0$ such that for every ball $B := B(x,r)$ and any function $f$ supported in $B$, we have:
\begin{equation} \label{Pbg}
\|f\|_{L^p(\Gamma)} \leq C r^{\beta\gamma} \|\nabla f\|_{L^p(\Gamma)}.
\end{equation}
Fix a point $x_0 \in \Gamma$, and for $r \in \mathbb{N} \setminus \{0\}$, define the function
\[ f_r(x) := \max\{0, r - \dist(x_0, x)\}. \]
Observe the following:
\begin{itemize}
\item $f_r$ is supported in $B(x_0,r)$;
\item $|f_r(x) - f_r(y)| \leq 1$ whenever $x \sim y$, which implies $\nabla f_r \leq 1$. Since $\nabla f_r$ is supported in $B(x_0,r)$, we have
\[ \|\nabla f_r\|_{L^p} \leq V(x_0, r)^{1/p}; \]
\item $f_r \geq r/2$ on $B(x_0, r/2)$, so
\[ \|f_r\|_{L^p} \geq \frac{r}{2} V(x_0, r/2)^{1/p}. \]
\end{itemize}
Inserting these bounds into \eqref{Pbg} yields:
\[ \frac{r}{2} V(x_0, r/2)^{1/p} \leq C r^{\beta\gamma} V(x_0, r)^{1/p}, \]
which simplifies to:
\[ r^{1 - \beta\gamma} \leq 2C \left( \frac{V(x_0, r)}{V(x_0, r/2)} \right)^{1/p} \lesssim 1 \]
by the doubling property \eqref{D}. Since this inequality must hold for all integers $r \geq 1$, it forces the exponent $1 - \beta\gamma \leq 0$, i.e., $\gamma \geq \frac{1}{\beta}$. The corollary follows.
\ep

\begin{lemma} \label{lemPp=2}
Let $(\Gamma,\mu)$ be an unbounded doubling graph such that the Markov semigroup is analytic and satisfies \eqref{UE}. Then the inequality (P$_{2,\gamma\beta}$) fails for any $\gamma < \frac{1}{2}$.
\end{lemma}

\bp
\textbf{Step 1:} We claim there exists $\Lambda \geq 1$ (independent of $y \in \Gamma$ and $k \in \mathbb{N}^*$) such that
\[
\sum_{x \in \Gamma \setminus B(y, \Lambda k^{1/\beta})} p_k(x,y) m(x) \leq \frac{1}{2}.
\]
Indeed, define $C_j(y,r) := B(y,2^j r) \setminus B(y,2^{j-1} r)$. Then:
\begin{align*}
\sum_{x \in \Gamma \setminus B(y, \Lambda k^{1/\beta})} p_k(x,y) m(x)
&= \sum_{j \in \mathbb{N}^*} \sum_{x \in C_j(y, \Lambda k^{1/\beta})} p_k(x,y) m(x) \\
&\lesssim \sum_{j \in \mathbb{N}^*} \frac{V(y, 2^j \Lambda k^{1/\beta})}{V(y, k^{1/\beta})} \exp\left(-c (2^{j-1}\Lambda)^{\frac{\beta}{\beta - 1}}\right) \\
&\lesssim \exp\left(-\frac{c \Lambda^{\frac{\beta}{\beta - 1}}}{2}\right) \sum_{j \in \mathbb{N}^*} (2^j \Lambda)^D \exp\left(-\frac{c (2^{j-1} \Lambda)^{\frac{\beta}{\beta - 1}}}{2} \right) \\
&\lesssim \exp\left(-\frac{c \Lambda^{\frac{\beta}{\beta - 1}}}{2} \right),
\end{align*}
where we used \eqref{UE} and the doubling property \eqref{defD}. The claim follows by choosing $\Lambda$ large enough.

\medskip

\noindent \textbf{Step 2:} There exists $\eta > 0$ such that the set
\[
S := \left\{ x \in B(y, \Lambda k^{1/\beta})\, , \,  p_k(x,y) \geq \frac{1}{4 V(y, \Lambda k^{1/\beta})} \right\}
\]
satisfies $m(S) \geq \eta V(y, k^{1/\beta})$. Indeed, for $x \in S$, by \eqref{UE}, $p_k(x,y) \leq C / V(y, k^{1/\beta})$. Combining with Step 1, we get:
\[
\frac{1}{2} \leq \sum_{x \in B(y, \Lambda k^{1/\beta})} p_k(x,y) m(x)
\leq \frac{1}{4V(y,\Lambda k^{1/\beta}} V(y,\Lambda k^{1/\beta} + C\frac{m(S)}{V(y, k^{1/\beta})},
\]
which gives $m(S) \geq \frac{1}{4C} V(y, k^{1/\beta})$.

\medskip

\noindent \textbf{Step 3:} Fix $y \in \Gamma$ and define:
\[
f_k(x) := \max\left\{ 0, p_k(x,y) - \frac{1}{k V(y, k^{1/\beta})} \right\}.
\]
By doubling, the ratio $V(y, \Lambda k^{1/\beta}) / V(y, k^{1/\beta})$ is uniformly bounded by some constant $k_0$, and for $k \geq k_0$, Step 2 implies:
\begin{equation} \label{lbonfk}
\|f_k\|_2^2 \gtrsim \frac{1}{V(y, k^{1/\beta})}.
\end{equation}
Also,
\begin{equation} \label{bdgk}
0 \leq f_k(x) - p_k(x,y) \leq \frac{1}{k V(y, k^{1/\beta})}.
\end{equation}
Thus, for $x \in \Gamma$:
\[
|\Delta f_k(x)| \leq |p_k(x,y) - p_{k+1}(x,y)| + |\Delta(f_k - p_k(\cdot,y))(x)| \lesssim \frac{1}{k V(y, k^{1/\beta})},
\]
where we use \eqref{dkUE} and the $L^\infty$-boundedness of $\Delta$. Therefore:
\begin{equation} \label{gradbound}
\|\nabla f_k\|_2^2 = \sum_{x \in \Gamma} [\Delta f_k](x) f_k(x) m(x) \lesssim \frac1{kV(y,k^{1/\beta})} \underbrace{\sum_{x\in \Gamma} p_k(x,y) m(x)}_{=1} \lesssim \frac{1}{k} \|f_k\|_2^2.
\end{equation}

Now observe that $f_k(x) = 0$ whenever
\[
p_k(x,y) \leq \frac{1}{k V(y, k^{1/\beta})} \leq C \frac{1}{V(y, k^{1/\beta})} \exp\left(- c \left( \frac{d(x,y)^\beta}{k} \right)^{\frac{1}{\beta-1}} \right),
\]
so
\[
d(x,y) \geq r_k := k^{1/\beta} \left( \frac{\log(Ck)}{c} \right)^{\frac{\beta-1}{\beta}}.
\]
Thus, $f_k$ is supported in $B(y, r_k)$ with $r_k \to \infty$, and
\[
\|\nabla f_k\|_2^2 \lesssim \frac{1}{k} \|f_k\|_2^2 \lesssim \frac{\log(r_k)^{\beta - 1}}{r_k^\beta} \|f_k\|_2^2 \lesssim \frac{1}{r_k^{2\gamma\beta + \epsilon}} \|f_k\|_2^2,
\]
whenever $\gamma < \frac{1}{2}$ and $0 < \epsilon < (1 - 2\gamma)\beta$. Hence, the Poincaré inequality (P$_{2,\gamma\beta}$) cannot hold if $\gamma < \frac{1}{2}$, completing the proof.
\ep

\begin{corollary} \label{Corg<L(b,p)}
Let $(\Gamma,\mu)$ be an unbounded doubling graph with non-degenerate probability transitions, such that the Markov semigroup is analytic and satisfies \eqref{UE}. Then \eqref{RRpgg} fails whenever $p \in (1,2)$ and 
\[
\gamma < \frac{1}{\beta} + \frac{1}{p} \left(1 - \frac{2}{\beta}\right).
\]
\end{corollary}

\bp
The key observation is that (P$_{\infty,1}$) holds. Indeed, since $\Gamma$ is connected, for any $x \in B(x_0, r)$, there exists a path of length at most $2r$ connecting $x$ to $\Gamma \setminus B(x_0, r)$. Therefore, for any $f \in L^\infty(\Gamma)$ supported in $B(x_0,r)$ and $x \in B(x_0, r)$:
\[
|f(x)| \leq 2r \sup_{y \sim z} |f(y) - f(z)| \leq 2r \sqrt{\frac{2}{\epsilon}} \sup_{y \in \Gamma} |\nabla f(y)|,
\]
where $\epsilon$ is from \eqref{defLBE}. Hence, $\|f\|_\infty \leq C r \|\nabla f\|_\infty$.

Now, suppose \eqref{RRpgg} holds for some $p \in (1,2)$ and $\gamma < \frac{1}{\beta} + \frac{1}{p} \left(1 - \frac{2}{\beta}\right)$. Then, by Remark \ref{RmkRR=>P}, (P$_{p,\gamma\beta}$) holds. Interpolating this and (P$_{\infty,1}$), using Theorem \ref{ThInterpolation}, we obtain (P$_{2,\delta}$) for
\[
\delta := \frac{p\gamma}{2} + \frac{1}{\beta} - \frac{p}{2\beta} < \frac{1}{2},
\]
which contradicts Lemma \ref{lemPp=2}. The result follows.
\ep

\section{Proof of Theorem \ref{MainThRM}} \label{Sproofm}

In this section, we revisit the results established in Section \ref{Sproofg} and explain how they can be adapted to the setting of Riemannian manifolds. Throughout this section, $(\Gamma,m)$ denotes a smooth, complete, connected, non-compact Riemannian manifold equipped with the geodesic distance, the Riemannian measure, the Laplace–Beltrami operator, and the Riemannian gradient (see Subsection \ref{Ssmanifold} for details).

\subsection{First Results}

\begin{lemma} \label{lemRp1m} 
Let $(\Gamma,m)$ be a Riemannian manifold satisfying the local gradient estimate on the heat kernel \eqref{LG} and the doubling condition \eqref{Dm}. Then the inequality (R$_{p,0}$) holds for all \( p \in [1,\infty] \); more precisely, for all \( f \in L^p(\Gamma) \) and \( x \in \Gamma \), we have
\begin{equation} \label{Rp0m} 
|\nabla e^{-\Delta} f(x)| \leq \sum_{j \in \mathbb{N}} 2^{jD} e^{-c4^j} \fint_{B(x,2^j)} |f| \, dm,
\end{equation}
where \( D \) is the ``dimension'' from \eqref{defDm}. By duality, the bound (RR$_{p,1}$) also holds for all \( p \in [1,\infty] \).
\end{lemma}

\begin{proof}
The lemma follows from \eqref{LG}. For \( j \in \mathbb{N} \), let \( C_0(x,1) := B(x,1) \) and \( C_j(x,1) := B(x,2^j) \setminus B(x,2^{j-1}) \) for \( j \geq 1 \). Then,
\begin{align*}
|\nabla e^{-\Delta} f(x)| 
&\leq \int_{\Gamma} |\nabla_x h_{1}(x,y)| |f(y)| \, dm(y) \\
&\leq \sum_{j \in \mathbb{N}} \int_{C_j(x,1)} |\nabla_x h_{1}(x,y)| |f(y)| \, dm(y) \\
&\leq \sum_{j \in \mathbb{N}} \sup_{y \in C_j(x,1)} |\nabla_x h_{1}(x,y)| \int_{B(x,2^j)} |f| \, dm \\
&\lesssim \sum_{j \in \mathbb{N}} \frac{1}{V(x,1)} e^{-c4^j} \int_{B(x,2^j)} |f| \, dm.
\end{align*}
By the doubling property \eqref{Dm}, or more precisely \eqref{defDm}, we have
\[
\frac{1}{V(x,1)} \lesssim 2^{jD} \frac{1}{V(x,2^j)},
\]
so that
\[
|\nabla e^{-\Delta} f(x)| \lesssim \sum_{j \in \mathbb{N}} 2^{jD} e^{-c4^j} \fint_{B(x,2^j)} |f| \, dm.
\]

After taking the \( L^p \)-norm, we use Minkowski's inequality, Hölder's inequality, and Fubini’s theorem to get:
\begin{align*}
\|\nabla e^{-\Delta} f\|_p 
&\lesssim \sum_{j \in \mathbb{N}} 2^{jD} e^{-c4^j} \left( \int_\Gamma \left| \fint_{B(x,2^j)} |f(y)| \, dm(y) \right|^p dm(x) \right)^{1/p} \\
&\leq \sum_{j \in \mathbb{N}} 2^{jD} e^{-c4^j} \left( \int_\Gamma \fint_{B(x,2^j)} |f(y)|^p \, dm(y) \, dm(x) \right)^{1/p} \\
&= \sum_{j \in \mathbb{N}} 2^{jD} e^{-c4^j} \left( \int_\Gamma |f(y)|^p \int_{B(y,2^j)} \frac{1}{V(x,2^j)} \, dm(x) \, dm(y) \right)^{1/p}.
\end{align*}
By the doubling property again,
\[
\int_{B(y,2^j)} \frac{1}{V(x,2^j)} \, dm(x) \lesssim 1,
\]
so
\[
\|\nabla e^{-\Delta} f\|_p \lesssim \sum_{j \in \mathbb{N}} 2^{jD} e^{-c4^j} \|f\|_p \lesssim \|f\|_p.
\]
The case \( p = \infty \) follows similarly, with an even simpler proof. This completes the proof.
\end{proof}

\begin{corollary} \label{Corg>1/p,1/2m}
Let $(\Gamma,m)$ be a Riemannian manifold satisfying \eqref{LG}, \eqref{Dm}, and \eqref{UEm}. Then the estimate \eqref{Rpgm} holds for all \( p \in (1,\infty) \) and \( 0 \leq \gamma \leq \min\{\frac{1}{2}, \frac{1}{p}\} \), and the estimate \eqref{RRpgm} holds for all \( p \in (1,\infty) \) and \( \max\{\frac{1}{2}, \frac{1}{p}\} \leq \gamma \leq 1 \).
\end{corollary}

\begin{proof}
Lemma \ref{lemRp1m} establishes that (R$_{p,0}$) holds for all \( p \in [1,\infty] \). The main result of \cite{CCFR15} asserts that \( \|\nabla f\|_p \lesssim \|\Delta^{1/2} f\|_p \) for all \( p \in (1,2] \), and thus also  \( \|\nabla e^{-\Delta} f\|_p \lesssim \|\Delta^{1/2} f\|_p \), since \( e^{-t\Delta} \) is contractive on \( L^p \). By interpolation (Theorem \ref{ThInterpolation}),
\[
\|\nabla e^{-\Delta} f\|_p \lesssim C_p \|\Delta^\gamma f\|_p
\]
for all \( p \in (1,\infty) \) and \( 0 \leq \gamma \leq \min\{\frac{1}{2}, \frac{1}{p}\} \).

The statement for \eqref{RRpgm} follows by duality (Theorem \ref{ThDuality}).
\end{proof}

\begin{lemma} \label{lemp=2m}
Let $(\Gamma,m)$ be a Riemannian manifold satisfying \eqref{Dm}. Then (RR$_{2,\gamma}$) holds if and only if \( \gamma \geq \frac{1}{2} \).
\end{lemma}

\begin{proof}
By \eqref{E2m} and analyticity of the heat semigroup,
\[
\|\nabla e^{-\Delta} f\|_2 = \|\Delta^{1/2} e^{-\Delta} f\|_2 \leq C \|f\|_2,
\]
so (R$_{2,0}$) holds. Interpolation with \eqref{E2m} yields (R$_{2,\gamma}$) for \( 0 \leq \gamma \leq \frac{1}{2} \). Duality (Theorem \ref{ThDuality}) then gives (RR$_{2,\gamma}$) for \( \frac{1}{2} \leq \gamma \leq 1 \).

To show failure of (RR$_{2,\gamma}$) for \( \gamma < \frac{1}{2} \), we show that 0 belongs to the \( L^2 \)-spectrum of \( \Delta \). Fix \( x_0 \in \Gamma \) and define \( f_r(x) := \max\{0, 1 - \frac{d(x,x_0)}{r} \} \). Then,
\[
\|\Delta^{1/2} f_r\|_2 = \|\nabla f_r\|_2 \leq \frac{1}{r} V(x_0,r) \lesssim \frac{1}{r} V(x_0, r/2)^{1/2} \leq \frac{2}{r} \|f_r\|_2,
\]
using the doubling property. Hence 0 lies in the spectrum of \( \Delta^{1/2} \) and thus in that of \( \Delta \). The conclusion follows as in the proof of Lemma \ref{lemp=2}.
\end{proof}

\begin{corollary} \label{Corg<1/p}
Let $(\Gamma,m)$ be a Riemannian manifold satisfying \eqref{LG}, \eqref{Dm}, and \eqref{UEm}. Then the estimate \eqref{RRpgm} fails whenever \( p \in (2,\infty) \) and \( 0 \leq \gamma < \frac{1}{p} \).
\end{corollary}

\begin{proof}
Assume for contradiction that (RR$_{p,\gamma}$) holds for some \( p \in (2,\infty) \) and \( \gamma < \frac{1}{p} \). Interpolating this with the bound (RR$_{1,1}$) from Lemma \ref{lemRp1m} yields (RR$_{2,\delta}$) for some \( \delta < \frac{1}{2} \), contradicting Lemma \ref{lemp=2m}.
\end{proof}

\subsection{Further results}

The following results are slightly more involved in the case of Riemannian manifolds than for graphs, due to the presence of the smoothing operator $e^{-t\Delta}$ in the definitions of \eqref{Rpgm} and \eqref{RRpgm}. On the other hand, we no longer need to assume \eqref{LG}. Since the proofs are similar to those presented in Subsection \ref{SsFG}, we omit some details and focus instead on the differences with the graph setting.

\begin{proposition} \label{prUE=>aFKm}
Let $(\Gamma,m)$ be a Riemannian manifold satisfying \eqref{Dm} and \eqref{UEm}. Then for any $p \in (1,\infty)$, any $\alpha \geq 0$, and any $\epsilon > 0$, there exists a constant $C_\epsilon = C(\alpha,p,\epsilon) > 0$ such that for any ball $B := B(x,r) \subset \Gamma$ with radius $r \geq 1$, and any function $f \in L^p(\Gamma,m)$ supported in $B$ satisfying $\|e^{-\Delta} f \|_p \geq \epsilon \|f\|_p$, we have
\begin{equation} \label{UE=>aFKm}
\|e^{-\Delta} f\|_{L^p} \leq C_\epsilon\, r^{\alpha\beta} \|\Delta^{\alpha} e^{-\Delta} f\|_{L^p}.
\end{equation}
In particular,
\[
\eqref{RRpgm} \implies \text{(P$_{p,\beta \gamma}$)}.
\]
\end{proposition}

\begin{proof}
We follow the strategy of Proposition \ref{prUE=>aFK}, based on an appropriate reproducing formula. Let $B = B(x,r)$ with $r \geq 1$, and let $f$ be as in the proposition. Set $u = \lambda^\beta r^\beta$.

\medskip
\noindent {\bf Step 1.} We claim that
\begin{equation} \label{claim1m}
\|e^{-(u+1)\Delta} f\|_p \leq \frac{1}{2} \|e^{-\Delta} f\|_p,
\end{equation}
for some sufficiently large $\lambda > 0$, independent of $f$ and $B$.
Indeed,
\begin{align} \label{UE->FK2m}
\|e^{-(u+1)\Delta} f\|_p &= \left( \int_{ \Gamma } \left| \int_{ \Gamma} h_{u+1}(x,y) f(y)\, dm(y) \right|^p dm(x) \right)^{1/p} \nonumber \\
&\leq \left( \int_{\Gamma} \sup_{y\in B} |h_{u+1}(x,y)|^p\, dm(x) \right)^{1/p} \|f\|_{L^1}.
\end{align}
By Hölder's inequality and the assumption on $f$, we have
\begin{equation} \label{UE->FK1m}
\|f\|_{L^1} \leq m(B)^{1 - \frac{1}{p}} \|f\|_p \leq \epsilon^{-1} m(B)^{1 - \frac{1}{p}} \|e^{-\Delta} f\|_p.
\end{equation}
We decompose the space by setting $C_0(\lambda B) := 2\lambda B$ and $C_j(\lambda B) := 2^{j+1}\lambda B \setminus 2^j \lambda B$ for $j \geq 1$. Observe that for $x \in C_j(\lambda B)$ and $y \in B$, we have
\[
d(x,y)^\beta \gtrsim (2^{j\beta} - 1)(u+1),
\]
which implies, using \eqref{UEm}, that
\[
\sup_{y\in B} |h_{u+1}(x,y)| \lesssim \frac{1}{m(\lambda B)} e^{-c 2^{j\beta/(\beta-1)}} \quad \text{for } x \in C_j(\lambda B).
\]
Thus,
\begin{align} \label{Lpbdht}
\int_{\Gamma} \sup_{y\in B} |h_{u+1}(x,y)|^p\, dm(x) 
&\lesssim \sum_{j \geq 0} m(2^{j+1}\lambda B) \left( \frac{e^{-c 2^{j\beta/(\beta-1)}}}{m(\lambda B)} \right)^p \nonumber \\
&\lesssim \sum_{j \geq 0} 2^{(j+1)D} m(\lambda B) \left( \frac{e^{-c 2^{j\beta/(\beta-1)}}}{m(\lambda B)} \right)^p \nonumber \\
&\lesssim \frac{1}{m(\lambda B)^{p-1}} \lesssim \frac{\lambda^{-(p-1)\delta}}{m(B)^{p-1}}.
\end{align}
By choosing $\lambda$ sufficiently large (depending on $\epsilon$), we ensure that
\[
\left( \int_{\Gamma} \sup_{y\in B} |h_{u+1}(x,y)|^p\, dm(x) \right)^{1/p} \leq \frac{1}{2} m(B)^{\frac{1}{p} - 1} \frac{\|e^{-\Delta} f\|_p m(B)^{1 - \frac{1}{p}}}{\|f\|_{L^1}} \leq \frac{1}{2} \frac{\|e^{-\Delta} f\|_p}{\|f\|_{L^1}},
\]
which, reinjected into \eqref{UE->FK2m}, gives \eqref{claim1m}.

\medskip
\noindent {\bf Step 2.} For $\alpha \in (0,1)$, we use the identity
\[
e^{-\Delta} f = e^{-(u+1)\Delta} f + \int_0^u \Delta e^{-(s+1)\Delta} f\, ds = e^{-(u+1)\Delta} f + \int_0^u \Delta^{1 - \alpha} e^{-s\Delta} \Delta^{\alpha} e^{-\Delta} f\, ds.
\]
Taking the $L^p$ norm gives
\[
\|e^{-\Delta} f\|_p \leq \|e^{-(u+1)\Delta} f\|_p + C \|\Delta^{\alpha} e^{-\Delta} f\|_p \int_0^u s^{\alpha - 1} ds \leq \frac{1}{2} \|e^{-\Delta} f\|_p + C r^{\alpha \beta} \|\Delta^{\alpha} e^{-\Delta} f\|_p,
\]
using analyticity \eqref{analyt} and $u \simeq r^\beta$. Rearranging gives the desired estimate \eqref{UE=>aFKm}.

\medskip
\noindent {\bf Step 3.} For $\alpha > 1$, choose an integer $n > \alpha$ and iterate the formula as in Step 4 of Proposition \ref{prUE=>aFK} to obtain
\[
e^{-\Delta} f = T(e^{-(u+1)\Delta} f) + \int_0^u \cdots \int_0^u \Delta^{n - \alpha} e^{-(s_1 + \cdots + s_n)\Delta} \Delta^{\alpha} e^{-\Delta} f\, ds_n \cdots ds_1,
\]
where $T$ is linear and bounded on all $L^p$ spaces. As in Step 1, we can choose $\lambda$ so that $\|T(e^{-(u+1)\Delta} f)\|_p \leq \frac{1}{2} \|e^{-\Delta} f\|_p$. Observe that 
\[
\int_0^u \cdots \int_0^u \frac{1}{(s_1 + \cdots + s_n)^{n - \alpha}}\, ds_n \cdots ds_1 \lesssim u^{\alpha} \simeq r^{\alpha \beta}
\]
so by using the analyticity of $e^{-\Delta}$ given in \eqref{analyt}, we have
\begin{multline*}
\left\| \int_0^u \cdots \int_0^u \Delta^{n - \alpha} e^{-(s_1 + \cdots + s_n)\Delta} \Delta^{\alpha} e^{-\Delta} f\, ds_n \cdots ds_1 \right\|_p \\ \lesssim \|\Delta^{\alpha} e^{-\Delta} f\|_p \int_0^u \cdots \int_0^u \frac{1}{(s_1 + \cdots + s_n)^{n - \alpha}}\, ds_n \cdots ds_1 \lesssim  r^{\alpha \beta} \|\Delta^{\alpha} e^{-\Delta} f\|_p.
\end{multline*}
We conclude that
\[
\|e^{-\Delta} f\|_p - \frac{1}{2} \|e^{-\Delta} f\|_p \lesssim r^{\alpha \beta} \|\Delta^{\alpha} e^{-\Delta} f\|_p.
\]
The proposition follows.
\end{proof}

\begin{lemma} \label{lemPp=2m}
Let $(\Gamma,m)$ be a Riemannian manifold satisfying \eqref{Dm} and \eqref{UEm}. Then the inequality \emph{(P$_{2,\gamma\beta}$)} fails whenever $\gamma < \frac{1}{2}$.
\end{lemma}

\begin{proof}
The first two steps are identical to those in Lemma \ref{lemPp=2}, so we only recall the conclusions.

\medskip

\noindent \textbf{Step 1.} There exists $\Lambda \geq 1$ (independent of $y \in \Gamma$ and $t \geq 1$) such that
\[
\int_{\Gamma \setminus B(y, \Lambda t^{1/\beta})} h_t(x,y) \, dm(x) \leq \frac{1}{2}.
\]

\medskip

\noindent \textbf{Step 2.} There exists $\eta > 0$ (independent of $t$ and $y$) such that the set
\[
S := S_{y,t} := \left\{ x \in B(y, \Lambda t^{1/\beta}) \,, \, h_t(x,y) \geq \frac{1}{4V(y,\Lambda t^{1/\beta})} \right\}
\]
satisfies $m(S) \geq \eta V(y, t^{1/\beta})$.

\medskip

\noindent \textbf{Step 3.} Fix $y \in \Gamma$ and define
\[
f_t(x) := \max\left\{0,\, h_t(x,y) - \frac{1}{tV(y, t^{1/\beta})} \right\}.
\]
By computations similar to those in \eqref{Lpbdht}, and using \eqref{Dm} and \eqref{UEm}, we have
\[
\|f_t\|_2^2 \lesssim \frac{1}{V(y, t^{1/\beta})}.
\]
We aim to show that this is still controlled (up to constants) by $\|e^{-t\Delta} f_t\|_2^2$. By the doubling condition \eqref{D}, the ratio $\frac{8V(y, \Lambda (t+1)^{1/\beta})}{V(y, t^{1/\beta})}$ is uniformly bounded for $t \geq t_0$, some $t_0 > 0$. For such $t$, we get:
\[
h_{t+1}(x,y) \geq \frac{1}{4V(y, \Lambda (t+1)^{1/\beta})} \quad \text{for all } x \in S_{y, t+1}.
\]
Since $f_t \geq h_t(\cdot,y) - \frac{1}{tV(y, t^{1/\beta})}$ and $e^{-t\Delta}$ is positive, it follows that
\[
e^{-t\Delta} f_t(x) \geq h_{t+1}(x,y) - \frac{1}{tV(y, t^{1/\beta})} \geq \frac{1}{8V(y, \Lambda (t+1)^{1/\beta})} \quad \text{on } S_{y, t+1}.
\]
Hence,
\[
\|e^{-t\Delta} f_t\|_2^2 \gtrsim \frac{1}{V(y, t^{1/\beta})},
\]
using Step 2 and the doubling property. Therefore,
\begin{equation} \label{prft}
\frac{1}{V(y, t^{1/\beta})} \lesssim \|e^{-t\Delta} f_t\|_2 \leq \|f_t\|_2 \approx \frac{1}{V(y, t^{1/\beta})},
\end{equation}
as desired.

\medskip

\noindent \textbf{Step 4.} We claim that $\Delta e^{-\Delta}$ is bounded on $L^p$ for all $p \in [1, \infty]$. The proof is standard and omitted; it mirrors that of Lemma \ref{lemRp1m}, using \eqref{dtUE} in place of \eqref{LG}.

Decompose $f_t = h_t(\cdot, y) - g_t$ with $\|g_t\|_\infty \leq \frac{1}{tV(y, t^{1/\beta})}$. Then,
\[
|\Delta e^{-\Delta} f_t(x)| \leq |\partial_t h_{t+1}(x,y)| + |\Delta e^{-\Delta} g_t(x)| \lesssim \frac{1}{tV(y, t^{1/\beta})}
\]
by \eqref{dtUE} and the $L^\infty$-boundedness of $\Delta e^{-\delta}$.
Combining this with \eqref{prft}, we deduce
\begin{equation} \label{lbonfk}
\|\nabla e^{-\Delta} f_t\|_2^2 = \int_{\Gamma} [\Delta e^{-\Delta} f_t] f_t \, dm \lesssim \frac{1}{tV(y, t^{1/\beta})} \cdot \int_\Gamma h_t(x,y) \, dm(x) \lesssim \frac{1}{t} \|e^{-t\Delta} f_t\|_2^2.
\end{equation}

\medskip

We now find a ball $B(y, r_t)$ such that $\supp f_t \subset B(y, r_t)$. As in Lemma \ref{lemPp=2}, using \eqref{UEm}, we obtain:
\[
r_t \lesssim t^{1/\beta} \ln^{(\beta-1)/\beta}(2t).
\]
Thus, we construct a sequence of functions $f_t$ supported in $B(y, r_t)$ with $r_t \to \infty$ and
\[
\|\nabla e^{-\Delta} f_t\|_2^2 \lesssim \frac{1}{t} \|f_t\|_2^2 \lesssim \frac{\ln(2r_t)^{\beta-1}}{r_t^\beta} \|f_t\|_2^2 \lesssim \frac{1}{r_t^{2\gamma\beta + \varepsilon}} \|f_t\|_2^2,
\]
for any $\gamma < \frac{1}{2}$ and $0 < \varepsilon < (1 - 2\gamma)\beta$. Hence, the Poincaré inequality \emph{(P$_{2, \gamma\beta}$)} fails when $\gamma < \frac12$.
\end{proof}

\begin{corollary} \label{Corg<L(b,p)m}
Let $(\Gamma, m)$ be a Riemannian manifold satisfying \eqref{Dm} and \eqref{UEm}. Then \eqref{RRpgm} fails whenever \( p \in (1,2) \) and 
\[
\gamma < \frac{1}{\beta} + \frac{1}{p} \left(1 - \frac{2}{\beta}\right).
\]
\end{corollary}

\begin{proof}
Since the manifold is connected, property \((P_{\infty,1})\) holds trivially. Assume, by contradiction, that \eqref{RRpgm} holds for some \( p \in (1,2) \) and
\[
\gamma < \frac{1}{\beta} + \frac{1}{p} \left(1 - \frac{2}{\beta}\right).
\]
Then, by Proposition \ref{prUE=>aFKm}, property \((P_{p,\gamma\beta})\) holds. Interpolating between \((P_{p,\gamma\beta})\) and \((P_{\infty,1})\) via Theorem \ref{ThInterpolation}, we obtain that \((P_{2,\delta})\) holds for
\[
\delta := \frac{p\gamma}{2} + \frac{1}{\beta} - \frac{p}{2\beta} < \frac{1}{2}.
\]
This contradicts Lemma \ref{lemPp=2m}, and the result follows.
\end{proof}

\begin{corollary} \label{Corg<1/bm}
Let $(\Gamma, m)$ be a Riemannian manifold satisfying \eqref{Dm} and \eqref{UEm}. Then \eqref{RRpgm} fails whenever \( p \in (1,\infty) \) and \( \gamma \in \left[0, \frac{1}{\beta}\right) \).
\end{corollary}

\begin{proof}
We argue by contraposition: assume that \eqref{RRpgm} holds, and we aim to show that \( \gamma \geq \frac{1}{\beta} \). By Proposition \ref{prUE=>aFKm}, the Poincar\'e inequality \((P_{\beta, \gamma})\) holds. That is, there exists a constant \( C > 0 \) such that for any ball \( B := B(x,r) \) with radius \( r \geq 1 \), and any smooth function \( f \) supported in \( B \), we have
\begin{equation} \label{Pbgm}
\|e^{-\Delta} f\|_{L^p(\Gamma)} \leq C r^{\beta\gamma} \|\nabla f\|_{L^p(\Gamma)}.
\end{equation}
Fix a point \( x_0 \in \Gamma \), and for \( r \geq 4 \), define
\[
f_r(x) := \max\{0, r - d(x_0, x)\}.
\]
Then:
\begin{itemize}
    \item \( |\nabla f_r| \leq \1_{B(x_0, 1)} \), hence
    \[
    \|\nabla f_r\|_{L^p} \leq V(x_0, r)^{1/p}.
    \]
    
    \item By Steps 1 and 2 of Lemma \ref{lemPp=2m}, there exist constants \( \Lambda \geq 1 \), \( \eta > 0 \) such that the set
    \[
    S_{y,1} := \left\{ x \in B(y, \Lambda)\, ,\,  h_1(x,y) \geq \frac{1}{4 V(y,\Lambda)} \right\}
    \]
    satisfies \( m(S_{y,1}) \geq \eta V(y,\Lambda) \). So for \( r \geq 4\Lambda \) and \( y \in B(x_0, r/2 - \Lambda) \), we get
    \[
    e^{-\Delta} f_r(y) \geq \int_{S_{y,1}} h_1(x,y) f_r(x) \, dm(x) \gtrsim \frac{m(S_{y,1})}{4V(y,\Lambda)} \inf_{B(y,\Lambda)} f_r(x) \gtrsim r,
    \]
 whence
    \[
    \|e^{-\Delta} f_r\|_p \gtrsim r V(x_0, r/2 - \Lambda)^{1/p} \gtrsim r V(x_0, r)^{1/p}
    \]
    by the doubling property.
\end{itemize}
Inserting these bounds into \eqref{Pbgm} gives
\[
r V(x_0, r)^{1/p} \lesssim r^{\beta\gamma} V(x_0, r)^{1/p},
\]
which simplifies to
\[
r^{1 - \beta\gamma} \lesssim 1.
\]
Since this holds for all \( r \geq 4\Lambda \), we must have \( 1 - \beta\gamma \leq 0 \), that is, \( \gamma \geq \frac{1}{\beta} \). The result follows.
\end{proof}

\bibliographystyle{amsalpha}
\bibliography{Biblio}


\end{document}